\documentclass{article}
\usepackage{graphicx}
\usepackage{amsmath}
\usepackage{amssymb}
\usepackage{booktabs}
\usepackage{multirow}
\usepackage{array}
\usepackage{amsfonts}
\usepackage{mathtools}
\usepackage{graphicx,pstricks,listings,stackengine}
\usepackage{ntheorem}
\usepackage{graphicx}
\usepackage{subfigure}
\usepackage{authblk}
\usepackage{todonotes}
\usepackage{hyperref}

\usepackage{bbm}
\usepackage{csquotes}
\usepackage{comment}

 \hypersetup{
    colorlinks=true,       
    linkcolor=red,          
     citecolor=blue,        
     urlcolor=cyan           
 }
 
\usepackage{cite}
\usepackage{float}
\usepackage{geometry}
\newtheorem{proposition}{Proposition}
\newtheorem{lemma}{Lemma}
\newtheorem{theorem}{Theorem}
\newtheorem*{proof}{Proof}
\newtheorem{corollary}{Corollary}
\newtheorem{remark}{Remark}
\newtheorem{definition}{Definition}

\title{Morse Index Classification and Landscape of Kuramoto System for Hebbian-based Binary Pattern Recognition}
 \author[a,b]{Xiaoxue Zhao}
 \author[c]{Xiang Zhou \thanks{Corresponding author: xizhou@cityu.edu.hk}}
  \affil[a]{School of Mathematics, Harbin Institute of Technology, Heilongjiang, China}
 \affil[b]{Department of Data Science, City University of Hong Kong, Hong Kong SAR}
  \affil[c]{Department of Mathematics, City University of Hong Kong, Hong Kong SAR}

 \date{}

\begin{document}

\maketitle

\begin{abstract}
   This study examines the Kuramoto model with a Hebbian learning rule and second-order Fourier coupling for binary pattern recognition. The system stores memorized binary patterns as stable critical points, enabling it to identify the closest match to a defective input. 
   However, the system exhibits multiple stable states and thus the dynamics are influenced by saddle points and other unstable critical points, which may disrupt convergence and recognition accuracy. 
   We systematically classify the stability of these critical points by analyzing the Morse index, which quantifies the stability of critical points by   the number of unstable directions.
   The index-1 saddle point is highlighted as the transition state on the energy landscape of the Kuramoto model.  These findings provide deeper insights into the stability landscape of the Kuramoto model than the stable equilibria, enhancing its theoretical foundation for binary pattern recognition.\\
    \par \textbf{Keywords:} Kuramoto system, 
    binary pattern recognition, saddle point, Morse index, energy landscape.
\end{abstract}


{\section{Introduction}}
The Kuramoto model  is a foundational framework for studying synchronization and collective dynamics in complex systems\cite{kuramoto1984chemical}, with applications ranging from neural networks to power grids and ecological systems
\cite{dong2013synchronization,dorfler2011critical,ha2013formation,choi2021asymptotic}. In recent developments, the model has been extended by incorporating the Hebbian learning rule and second-order Fourier coupling, enabling its application to binary pattern recognition\cite{nishikawa2004oscillatory,nishikawa2004capacity}.
The Hebbian learning rule, proposed by the Canadian psychologist Donald Olding Hebb in 1949 \cite{hebb1949the},  has its core idea as ``the synchronous activation between neurons will strengthen the connections between them".
As a cornerstone of neuroscience,  the Hebbian learning rule  governs how synaptic connections adapt based on neuronal activity, enabling the storage and retrieval of patterns \cite{Hopfield1982neural,Hopfield1984neurons,Duannurev2025}. By integrating this rule into the Kuramoto model, the system is capable of storing binary patterns and retrieving the closest match to a defective input, mimicking memory and recognition processes in neural systems\cite{hölzel2015stability}.

As an example of  the Kuramoto model with the Hebbian learning rule for the binary pattern recognition,  we  consider   the following 
system \cite{hölzel2015stability,nishikawa2004oscillatory,nishikawa2004capacity} consisting of $N$ Kuramoto-type oscillators, where the equation of motion for the $i$-th oscillator is
\begin{equation}\label{mod}
{\dot \varphi}_i =\frac{1}{N}\sum_{j=1}^{N}
C_{ij}\sin(\varphi_{j} - \varphi_i)+\frac{\varepsilon}{N}\sum_{j=1}^{N}\sin2(\varphi_{j}-\varphi_{i}) , ,, i\in [N]:=\{1, 2, \dots, N\}.
\end{equation}
Here, the state \(\varphi_i\) is a function of time $t$ and  \(\dot{\varphi}_i\) is its time derivative.
$\varepsilon$ is a positive constant. 
The connection weight matrix is an application of the Hebbian learning rule:
\begin{equation}
\label{C} C_{ij}:=\sum_{k=1}^{M}\xi_{i}^{k}\xi_{j}^{k},\;\; i,j\in [N]\quad \text{i.e.,}\quad \boldsymbol{C}:=[C_{ij}]=\sum_{k=1}^{M}\boldsymbol{\xi}^k(\boldsymbol{\xi}^k)^{\top},
\end{equation}
which encodes the $M$  {\it memorized (stored) binary patterns} $\{\boldsymbol{\xi}^1,\boldsymbol{\xi}^2,\dots,\boldsymbol{\xi}^M\}\subset\{-1,1\}^N$, where $\xi^k_i$ is the $i$-th component of $\boldsymbol{\xi}^k$,  $k\in [M]:=\{1,2,\dots,M\}$ and $i\in [N]$.   \(\boldsymbol{C}\) is a symmetric matrix and does not change when \(\boldsymbol{\xi}^k\) is replaced by \(-\boldsymbol{\xi}^k\). Any two   binary patterns $\boldsymbol{\xi}^k$ and $\boldsymbol{\xi}^l$  are called {\it distinct} if   $\boldsymbol{\xi}^k\neq \boldsymbol{\xi}^l$  
and  $\boldsymbol{\xi}^k\neq -\boldsymbol{\xi}^l$.
Binary patterns other than memorized binary patterns are called {\it non-memorized binary patterns}.

In system \eqref{mod},  the $M$   stored patterns $(\boldsymbol{\xi}^k)_{k\in [M]}$ are given for retrieval. When an input signal is not far away from one of the $M$ stored patterns, for example,
a noisy perturbation of $\boldsymbol{\xi}^k$, the network system \eqref{mod}   converges to a stable state encoding $\boldsymbol{\xi}^k$, e.g.,  $\arccos( \boldsymbol{\xi}^k)$,  via phase locking, thus achieving pattern retrieval. 
For the   model 
without  the second order term, \(\varepsilon = 0\),
if the  $M$ memorized patterns are mutual orthogonal,   the   steady-states of  \eqref{mod}  for these patterns are degenerate attractive limit sets, indicating  structural instability
\cite{hölzel2015stability}; in contrast, 
without this orthogonality condition, the memorized patterns become unstable whenever $M>2$ \cite{AonishiPRE1998}.

The second-order Fourier term \(\sin 2(\varphi_j - \varphi_i)\)  is necessary for   the system to support  isolated multi-stable   states  that correspond to the memorized binary patterns, which is crucial for pattern recognition \cite{nishikawa2004capacity,nishikawa2004oscillatory,zhao2020stability}. This second-order term typically  enhances the stability of critical points by modifying the system's energy landscape to ensure that the stored patterns   are stable and can be reliably retrieved by choosing the proper value of $\varepsilon$, meanwhile 
it also influences the Morse index of all critical points, which determines their role in the system's dynamics (e.g., attractors vs. saddle points).
For example,  when $\varepsilon$ is sufficiently large, {\it all} binary patterns, both memorized and non-memorized, are asymptotically stable, regardless of mutual orthogonality condition \cite[Theorem 3.2]{zhao2020stability} ---  an undesired scenario  to   retrieve   only the memorized patterns. 
And if the mutual orthogonality condition fails, certain  memorized binary patterns may lose the stability for  a small   $\varepsilon$ \cite{li2022hebbian}. 
Therefore,  it is important to use the proper size of $\varepsilon$ \cite{nishikawa2004oscillatory}
or  impose  the orthogonality conditions of $(\boldsymbol{\xi}^k)$. However, in general, it is impossible to rule out the existence of stable, non-memorized binary patterns when $M > 2$. Refer to \cite{zhao2020stability,li2022hebbian,zhao2023unified,zhao2025first, zhao2025binary} for more details.

 The existing  works mentioned above have extensively studied how to characterize the stable equilibria, and some of them \cite{hölzel2015stability,zhao2020stability} extended to study the points 
on the line between these binary patterns.
However, since the system exhibits multiple stable states, 
the associative memory mechanism  critically relies on the basins
of attraction of distinct stable memories for the input signal.
Therefore, besides identifying  the stable equilibria,
it is significant  to  further 
take a holistic view about the energy landscape \cite{Energylanscapes,Energylanscapes2024ReviewJCP} of the system \eqref{mod}, particularly by analyzing the roles of non-stable critical points, such as saddle points and unstable equilibria in shaping the basin of attraction and the system's dynamics.  The Kumamoto oscillators 
\eqref{mod}, as well as many other examples in physics, chemistry\cite{Energylanscapes}, and deep learning \cite{LossPDE2023,TUPNAS2024Review},
poses   multiple stable critical points (or local minimum points)
and exhibits very rich dynamics in various settings, ranging from phase transition\cite{npjLeiZhang2016},  noise-induced transitions\cite{FW2012} to global optimization \cite{SAOPT}.
Even though a multitude of advanced algorithms for the exploration and exploitation of such landscapes have been developed (e.g. \cite{String2002,GAD2011,yin2019high,TPT_review,Energylanscapes2024ReviewJCP}), to theoretically study specific models remains a challenging task.  For example, to explore the basin of attraction for the oscillatory system \eqref{mod} for each stable critical point is specially demanding and most works are based on numerical experiments\cite{nishikawa2004oscillatory}.

The unstable critical points are not merely mathematical artifacts; they play a pivotal role in determining the system's trajectories by acting as transitional states or barriers. 
The prevailing concept of {\it transition state}     on an energy landscape in chemistry \cite{Erying,Wigner,Pollak2005Rev} is the lowest-energy point on the basin  boundary, representing the reaction barrier from reactants to products. It is an unstable critical state with Morse index 1, connecting two neighboring stable equilibria (``reactant'' and ``product'') by its one-dimensional unstable manifold.  Its implication in the stochastic dynamics can be found in \cite{FW2012} for rigorous justification.
These concepts about energy landscape and transition states can also benefit our outstanding of the models
for associative memory.  For example, saddle points can influence whether the system successfully retrieves a stored pattern or diverges to an incorrect state.

In this paper, our objective is to conduct a comprehensive study on the stability and Morse index classification of the critical points of the system  \eqref{mod}, with arbitrary number of oscillators $N$ and two memorized binary patterns, i.e., $M=2$. 
The reason for $M=2$  is that
even  when $M$ is as small as $3$,
without assuming the orthogonality condition, 
the system \eqref{mod} can not maintain the stability of the memorized binary patterns at a small $\varepsilon$, and meanwhile generate more and more {\it stable} non-memorized binary-pattern critical points as $\varepsilon$ increases. Refer to Section \ref{sec2} and \cite{li2022hebbian,L-Z-Z}.
In addition,  there is a practical strategy \cite{zhao2023unified,zhao2025first} of reducing the recognition process of $M$ memorized binary patterns to $M/2$ groups of recognition processes for two memorized patterns.
By focusing solely on $M = 2$ memories, without imposing the orthogonality constraint between them, we can achieve a highly flexible model while maintaining theoretical clarity.
We summarize  our main contributions below.
\begin{itemize}
	\item 
	For any 
	(unstable) non-memorized binary patterns, $\boldsymbol{\varphi}\in \{0,\pi\}^N$, 
	Theorem \ref{thm1} demonstrates that we can completely determine its Morse index by examining only the cardinality of certain (component) index sets which is straightforward  to count.
	
	\item 
	By expanding the binary patterns to ternary patterns, where $\boldsymbol{\varphi} \in \{0, \pi/2, \pi\}^N$, we provide a necessary and sufficient condition for critical points (Proposition \ref{prop}), derive the spectrum (Theorem \ref{thm-mid}), and establish that the index-$1$ saddle point corresponds to the midpoint of two memories (Theorem \ref{Morthm1}).
	\item We illustrate our results by an example of an $N=6$ oscillator exhibiting   bi-stability,  identify all  Morse index-1 saddle points, and find the transition state with the lowest energy.
\end{itemize}

This paper is organized as follows. In Section \ref{sec2}, basic concepts and preliminaries of the Kuramoto system \eqref{mod} are introduced.
Section \ref{subsec3.2} and Section \ref{sec-mid} deal with critical points of binary patterns and ternary patterns, respectively.  Section \ref{sec5} demonstrates the results of the first two parts through an example.
Finally, a brief conclusion is given in Section \ref{secon}.

\section{Preliminaries}\label{sec2}

This section introduces the energy function,   critical points and the Morse index, and establishes the mathematical framework and theoretical tools required for the study. 

\subsection{Energy function, critical points, Morse index}
Let $\boldsymbol{\varphi}=(\varphi_1,\varphi_2,\dots,\varphi_N)^{\top}$ be the collection of totally $N$ oscillators.
The connection weight matrix $\boldsymbol{C}$ is symmetric, so  \eqref{mod} is a first-order gradient flow $\dot{\boldsymbol{\varphi}}=-\nabla V(\boldsymbol{\varphi})$ associated with the    potential energy  
\begin{equation} \label{V}
V(\boldsymbol{\varphi})=-\frac{1}{2N}\sum_{i,j=1}^{N}C_{ij}\cos(\varphi_j-\varphi_i)-\frac{\varepsilon}{4N}\sum_{i,j=1}^{N}\cos 2(\varphi_j-\varphi_i).
\end{equation} 
$V$ is an even function $V(\boldsymbol{\varphi})=V(-\boldsymbol{\varphi})$ and   satisfies the translation invariance $V(\boldsymbol{\varphi})=V(\boldsymbol{\varphi}+\delta \mathbbm{1} )$ for any real number $\delta$, where 
\(\mathbbm{1}=(1,1,\dots,1)^{\top}\in\mathbb{R}^N\).
So $\nabla V(\boldsymbol{\varphi}) \cdot \mathbbm{1}\equiv 0$ and consequently the flow \eqref{mod} satisfies 
$\sum_{i=1}^N {\varphi}_i (t) \equiv \sum_{i=1}^N {\varphi}_i (0)$ for all $t$.
If $\boldsymbol{\varphi}(t)$ satisfies  \eqref{mod}, then
$-\boldsymbol{\varphi}(t)$ satisfies  \eqref{mod} too. Thus, if any point $\boldsymbol{\varphi}\in \mathbb{R}^N$
is a critical point, then 
$-\boldsymbol{\varphi}$, as well as  $ \boldsymbol{\varphi} \operatorname{mod} 2\pi$,
is also a critical point.

Let  \(\boldsymbol{f}(\boldsymbol{\varphi})=(f_1(\boldsymbol{\varphi}),f_2(\boldsymbol{\varphi}),\dots,f_N(\boldsymbol{\varphi})^{\top}\in \mathbb{R}^N\) be 
the right-hand side of system \eqref{mod}. 
The Jacobian matrix $\boldsymbol{\mathcal{J}}_{\boldsymbol{\varphi}}=[\mathcal{J}_{ij}]$ of \eqref{mod}, which is the negative Hessian matrix of the potential function $V$, at an arbitrary point \(\boldsymbol{\varphi}\), is:
\begin{align}
\mathcal{J}_{ij}&:=\frac{f_i(\boldsymbol{\varphi})}{\partial \varphi_j}=\frac{1}{N}C_{ij}\cos(\varphi_j-\varphi_i)+\frac{2\varepsilon}{N}\cos2(\varphi_j-\varphi_i),\quad i,j\in [N] \;\; \text{and}\;\; i\neq j, \label{Jaco1}\\
\mathcal{J}_{ii}&:=\frac{f_i(\boldsymbol{\varphi})}{\partial \varphi_i}=-\frac{1}{N}\sum_{j=1,j\neq i}^{N}C_{ij}\cos(\varphi_j-\varphi_i)- \frac{2\epsilon}{N} \sum_{j=1,j\neq i}^{N}\cos2(\varphi_j-\varphi_i) ,\quad i\in [N].  \label{Jaco2}
\end{align}
There is a trivial zero eigenvalue and the eigenvector $\mathbbm{1}$, i.e., 
\(\boldsymbol{\mathcal{J}}_{\boldsymbol{\varphi}}\cdot\mathbbm{1}=0\cdot\mathbbm{1}\); this is  the consequence of  the translational invariance. 

As stated in \cite[Lemma 2.4]{monzon2005global}, if the zero eigenvalue is simple, then we can analyze the stability of the critical points of system \eqref{mod} through the spectrum of the Jacobian matrix.

We next define the Morse index and stability of critical points  of system \eqref{mod}, by excluding the trivial zero eigenvalue with the consensus 
direction $\mathbbm{1}$.
\begin{definition}\label{def1}
	A critical point \(\boldsymbol{\varphi}^*=(\varphi_1^*,\varphi_2^*,\dots,\varphi_N^*)^{\mathrm{T}}\) of system \eqref{mod} is said to be of  (Morse) index-$k$ if the number of positive eigenvalues of the Jacobian matrix at  \(\boldsymbol{\varphi}^*\) is $k$.  It is referred to as attractive or asymptotically stable if its Morse index is zero and the zero eigenvalue is simple, i.e., 
	all remaining $N-1$ eigenvalues are negative.
\end{definition}

The Morse index is the dimension of the largest subspace of the tangent space 
at a critical point on which the Hessian is negative definite (i.e, the Jacobian is positive definite).
When \(0 < k < N - 1\), the critical point \(\boldsymbol{\varphi}^*\) is a saddle point of system \eqref{mod}. For $k=0$ (or $N-1$), \(\boldsymbol{\varphi}^*\) is a minimum (or maximum) point of the energy function of the system \eqref{V}.  The translation $\boldsymbol{\varphi}^*\to \boldsymbol{\varphi}^* + \delta \mathbbm{1}$ does not change the  Morse index.

\subsection{Binary-pattern critical points of Kumamoto oscillators   }
There are  $2^N$ special  fixed-point solutions to the system \eqref{mod} corresponding to all possible binary patterns of length $N$.
Let  $\boldsymbol{\eta}=(\eta_{1},\eta_{2},\dots,\eta_{N})^T$ be an any $N$-dimensional vector of $1$'s and $-1$'s representing a binary pattern, then there is 
a  critical point  of  \eqref{mod} corresponding to $\boldsymbol{\eta}$, satisfying 
\begin{equation}\label{Hebbian}
\left|\varphi_{i}^*(\boldsymbol{\eta})-\varphi_{j}^*(\boldsymbol{\eta})\right|=\begin{cases}0, & \text{ if }\eta_{i}=\eta_{j};\\\pi,& \text{ if }\eta_{i}\ne\eta_{j}. \end{cases}
\end{equation}
The condition \eqref{Hebbian} is sufficient to make the right hand side of \eqref{mod} vanish, so it is a critical point encoding  the    binary pattern $\boldsymbol{\eta}$.
Among all points by  translational invariance (global phase shift), 
we denote the   solution of \eqref{Hebbian} in   $ \{0,\pi\}^N$
by 
$\boldsymbol{\varphi}^*(\boldsymbol{\eta})=(\varphi_{1}^*(\boldsymbol{\eta}),\varphi_{2}^*(\boldsymbol{\eta}),\dots,\varphi_{N}^*(\boldsymbol{\eta}))^{\top}$
where  
$$\varphi^*_i(\boldsymbol{\eta}):=\arccos \eta_i$$ for $i\in [N]$. 
Note that 
$ \arccos (-\boldsymbol{\eta})=\pi-\arccos \boldsymbol{\eta}$
is also a solution of \eqref{Hebbian} in   $ \{0,\pi\}^N$.   
Any point in   $ \{0,\pi\}^N$ (subject to translation invariance)  is referred  to     as the {\it binary-pattern critical point} of \eqref{mod}; 
and  the associated binary pattern can be  reconstructed by 
$\boldsymbol{\eta}=\cos(\boldsymbol{\varphi})$
or $\boldsymbol{\eta}=-\cos(\boldsymbol{\varphi})$.
We do not need to distinguish the binary-pattern critical points
$ \arccos\boldsymbol{\eta}$ and  $ \arccos(-\boldsymbol{\eta})$, as there is no need to distinguish a binary pattern $\boldsymbol{\eta}$ and its inverted image
$-\boldsymbol{\eta}$. 

So, there are effectively a total of  \(2^{N - 1}\)  binary-pattern critical points of   \eqref{mod} taking values in the space  $\{0, \pi\}^N$,   by choosing the first component  $\varphi^*_1(\boldsymbol{\eta})=0$ without loss of generality. 
We   exchange the use of a binary pattern $\boldsymbol{\eta}$ and its corresponding critical point $\boldsymbol{\varphi}^*(\boldsymbol{\eta})$, whenever there is no ambiguity.

\begin{definition}\label{def2}
	A binary pattern \(\boldsymbol{\eta}\) is said to be of index-$k$ if the binary-pattern critical point $\boldsymbol{\varphi}^*(\boldsymbol{\eta})$ of system \eqref{mod} that satisfies equation \eqref{Hebbian} corresponding to $\boldsymbol{\eta}$ is of index-$k$.
\end{definition}

At a binary-pattern critical point, the energy function $V$ takes the form 
$    V(\boldsymbol{\varphi}^*(\boldsymbol{\eta}))=-\frac{1}{2N}\sum_{k=1}^M  (\boldsymbol{\eta} 
\cdot \boldsymbol{\xi}^k)^2-\frac{N\varepsilon}{4}$.
And the Jacobian matrix becomes 
$\mathcal{J}_{ij}  =\frac{1}{N} \sum_{k=1}^M \xi^k_i\xi^k_j\eta_i\eta_j  +\frac{2\varepsilon}{N}$ and $
\mathcal{J}_{ii}=-\frac{1}{N}\sum_{j=1,j\neq i}^{N} \sum_{k=1}^{M}\xi^k_i\xi^k_j\eta_i\eta_j + \frac{2\varepsilon}{N}- 2\varepsilon$ for $i,j\in [N]$.
For any binary-pattern critical point, 
if $\lambda$ is a non-zero eigenvalue of the symmetric matrix $\mathcal{J}^0$ associated with    $\varepsilon=0$, then $\lambda-2\varepsilon$ becomes  an eigenvalue of the Jacobian matrix $\mathcal{J}$ for any $\varepsilon$ \cite[Theorem 2]{nishikawa2004oscillatory}. 


\subsection{Stability and spectrum of   binary-pattern critical points}\label{subsec-binary}
If the stored $M$ binary patterns of system \eqref{mod} are {\it orthogonal} to each other, i.e.$\sum_{i} \xi^k_i \xi^l_i=N\delta_{kl}$,  their Morse indices are zero for $\varepsilon>0$. In fact, the following result \cite[Theorem 4.1]{zhao2020stability}   is  for any real number $\varepsilon$.

\begin{displayquote}
	Let $M\ge 2$ and the
	$M$ memorized binary patterns $\{\boldsymbol{\xi}^k\}_{k\in [M]}$ in \eqref{mod}   be mutually orthogonal. Let $\varepsilon$ 
	be any real number. Then for each $k\in [M]$, at each stored  binary-pattern critical point  $\boldsymbol{\varphi}^*(\boldsymbol{\xi}^k)$ defined in \eqref{Hebbian},  the spectrum of the Jacobian matrix is 
	\[\underbrace{-1-2\varepsilon,\dots,-1-2\varepsilon}_{N-M}, \;\;\underbrace{-2\varepsilon,\dots,-2\varepsilon}_{M-1},\;\;0. \]
\end{displayquote}
So, if $\varepsilon>0$, then $\boldsymbol{\varphi}^*(\boldsymbol{\xi}^k)$ is of index-$0$ and asymptotically stable for all $k\in [M]$.
This result   clearly  demonstrates the importance of the second order Fourier term in \eqref{mod}, which can guarantee the 
asymptotical stability of the stored $M$ patterns when $\varepsilon>0$\cite{nishikawa2004oscillatory}. 
Under the same orthogonality condition, the stability of other non-memorized binary pattern depends on the size of $\varepsilon$ and the ``distance'' to the $M$ stored binary patterns and 
in many cases, the  non-memorized binary patterns can be also stable for   $\varepsilon>0$ \cite[Theorem 4.2]{zhao2020stability}.

But the limitation of these results in \cite{hölzel2015stability, zhao2020stability} is the orthogonality condition of the stored binary patterns, which is not easy to satisfy  in most applications. If we assume no such condition, then the memorized binary-pattern critical points appear to be typically unstable when $\varepsilon=0$ \cite{AonishiPRE1998}. Even though  
a positive $\varepsilon$ may help increase the error-free  capacity, which is the ratio between the number of stable stored patterns  and the number of neurons $N$,   the stability of memorized 
binary-pattern can not be maintained 
for every $\varepsilon>0$ if $M>2$.
More detailed and comprehensive  empirical investigation can be found in \cite{nishikawa2004capacity,nishikawa2004oscillatory, 6879338}.

We have established some rigorous   results  when $M=2$  in our previous work \cite{L-Z-Z} and review them  below, which will be also used in Section \ref{subsec3.2}.  These theoretic results are restricted to $M=2$, but the   orthogonality condition is not required. 
From here on, we shall only focus on $M=2$ with two arbitrary distinct stored memories  $\{\boldsymbol{\xi}^1,\boldsymbol{\xi}^2\} $ in system \eqref{mod}.

We partition the whole   set $[N]=\{1,2,\cdots,N\}$ of  the $N$ oscillators into two subsets, labeled by 
$I_1$ and $I_2$:  
\begin{equation}\label{index}I_1:=\left\{i\in  [N]\,\mid\, \xi^1_i=\xi^2_i\right\}\,\,\, \text{and} \,\,\, I_2:=\left\{i\in [N]\,\mid\, \xi^1_i\neq\xi^2_i\right\}.\end{equation}
$I_1\cup I_2=[N]$ and $I_1 \cap I_2=\emptyset$.
Moreover, by $\boldsymbol{\xi}^1\neq \pm\boldsymbol{\xi}^2$, we have $I_1\neq \emptyset$ and $I_2\neq \emptyset$, where  $|I_1|$ refers to the cardinality of a set $I_1$. Then  $|I_1|=N-|I_2|$,  taking values in 
$\{1,2,\ldots, N-1\}$.
For two orthogonal binary patterns, 
$|I_1|=|I_2|=N/2$. 
So the distance from $|I_1|$ (or $|I_2|$) to $N/2$ is a heuristic  indicator of how far away these  two binary patterns.

By defining a critical second-order coupling strength \begin{equation}\label{epstar}
\varepsilon^*:=\min\Big\{\frac{|I_1|}{N},\frac{|I_2|}{N}\Big\},
\end{equation}
which is less than or equal $1/2$,
the following theorem \cite[Theorem 2.1]{L-Z-Z} shows the condition of $\varepsilon$ for  the stability of memorized    and non-memorized binary patterns.

\begin{theorem}(\cite{L-Z-Z})\label{lem2}
	Let $\{\boldsymbol{\xi}^1,\boldsymbol{\xi}^2\}\subset\{-1,1\}^N$ be two memories in   \eqref{mod} {satisfying $1<|I_1|<N-1$}. 
	\begin{enumerate}
		\item[$(1)$] If $\varepsilon>0$,  then the memorized patterns $\pm\boldsymbol{\xi}^1$ and $\pm\boldsymbol{\xi}^2$ are asymptotically stable.
		\item[$(2)$]  If $0<\varepsilon<\varepsilon^*$,  then all   binary-pattern critical points
		are  unstable except for  $ \pm\boldsymbol{\xi}^1$ and $\pm\boldsymbol{\xi}^2$; 
		\item[$(3)$]  If $\varepsilon> \varepsilon^*$,  then there exists at least a   non-memorized binary pattern    $ \boldsymbol{\eta}\in \{-1,1\}^N\setminus\{\pm\boldsymbol{\xi}^1,\pm\boldsymbol{\xi}^2\}$ that is asymptotically stable.
	\end{enumerate}
\end{theorem}
\begin{remark}
	In general, the memorized patterns are not very close to each other, therefore neither $|I_1|$ nor 
	$|I_2|$ is close to one. But in case  \(|I_1| = 1\) or \(N - 1\), there are also results analogous to Theorem \ref{lem2}. Interested readers may refer to \cite[Proposition 2.2]{L-Z-Z}.
\end{remark}

The case $(2)$ in Theorem \ref{lem2} is the preferred situation for error-free retrieval of the two stored binary patterns, since these two are global attractive and there is no   non-memorized binary pattern which is stable.
This is the standard bi-stable system with a double-well potential. 
But for   $\varepsilon>\varepsilon^*$, an external signal may converge to new binary pattern, failing to  retrieve the stored memories. The energy landscape when $\varepsilon> \varepsilon^*$ is more complicated with more stable patterns emerging.

\section{Morse Index classification of     binary-pattern critical points}\label{subsec3.2}

We focus on the detailed analysis about the Morse index of all  binary-pattern critical points 
of \eqref{mod} with the   $M=2$ memories when   $\varepsilon\in (0,\varepsilon^*)$ with $\varepsilon^*$ defined in \eqref{epstar}. 
We only need to study the index classification of non-memorized binary pattern critical points, since \cite[Theorem 2.1]{L-Z-Z} has shown  the  memorized binary-patterns $\boldsymbol{\xi}^1$ and $\boldsymbol{\xi}^2$ are
asymptotically stable in this range of $\varepsilon$.

\begin{theorem}\textbf{(Morse index of (unstable) non-memorized binary patterns)}\label{thm1}
	Let $\{\boldsymbol{\xi}^1,\boldsymbol{\xi}^2\}\subset\{-1,1\}^N$ be two memories in system  \eqref{mod} and the corresponding  index sets $I_1,I_2$ be defined in \eqref{index}.
	
	For each  non-memorized binary-pattern $\boldsymbol{\eta}\in \{-1,1\}^N\setminus \{\pm\boldsymbol{\xi}^1,\pm \boldsymbol{\xi}^2\}$,  
	define the following partition of $I_1$ and $I_2$ according to $\boldsymbol{\eta}$:
	\begin{equation} \label{defI}
	\begin{split}
	I_{11}&:=\{j\in I_1\mid \xi^1_j=\xi^2_j=\eta_j\},\quad \quad  I_{12}:=\{j\in I_1\mid \xi^1_j=\xi^2_j=-\eta_j\} ,\\
	I_{21}&:=\{j\in I_2\mid \xi^1_j=-\xi^2_j=\eta_j\},\quad   I_{22}:=\{j\in I_2\mid \xi^1_j=-\xi^2_j=-\eta_j\}.
	\end{split}
	\end{equation}
	Assume \(0<\varepsilon<\varepsilon^*\) with \(\varepsilon^*\) defined in    \eqref{epstar}. 
	
	Then 
	the Morse index of $\boldsymbol{\eta}$ is determined by the cardinalities of these index sets \eqref{defI} in the following way:
	\begin{enumerate}
		\item[$(1)$]    $\min\{|I_{11}|,|I_{12}|\}=0$: 
		\begin{enumerate}
			\item[$(1a)$]  if $\min\{|I_{21}|,|I_{22}|\}=1$, then the index of $\boldsymbol{\eta}$ is $1$;
			\item[$(1b)$]  if $\min\{|I_{21}|,|I_{22}|\}\ge 2$ and $\big||I_{21}|-|I_{22}|\big|<\varepsilon N$,  then $\boldsymbol{\eta}$ is then the Morse index of $\boldsymbol{\eta}$ is $1$;
			\item[$(1c)$]  if $\min\{|I_{21}|,|I_{22}|\}\ge 2$ and $\big||I_{21}|-|I_{22}|\big|>\varepsilon N$, then the Morse index of $\boldsymbol{\eta}$ is  $$\min\{|I_{21}|,|I_{22}|\} \ge 2.$$
		\end{enumerate}
		\item[$(2)$]  $\min\{|I_{11}|,|I_{12}|\}=1$ or ``$\min\{|I_{11}|,|I_{12}|\}\ge 2$ and $\big||I_{11}|-|I_{12}|\big|<\varepsilon N$'':
		\begin{enumerate}
			\item[$(2a)$] if $\min\{|I_{21}|,|I_{22}|\}=0$, then the Morse index of $\boldsymbol{\eta}$ is $1$; 
			\item[$(2b)$] if $\min\{|I_{21}|,|I_{22}|\}=1$, then the Morse index of $\boldsymbol{\eta}$ is  $2$;
			\item[$(2c)$] if $\min\{|I_{21}|,|I_{22}|\}\ge 2$ and $\big||I_{21}|-|I_{22}|\big|<\varepsilon N$, then the Morse index of $\boldsymbol{\eta}$ is $2$;
			\item[$(2d)$] if $\min\{|I_{21}|,|I_{22}|\}\ge 2$ and $\big||I_{21}|-|I_{22}|\big|>\varepsilon N$, then  the Morse index of $\boldsymbol{\eta}$ is  $$\left(\min\{|I_{21}|,|I_{22}|\}+1\right)\ge 3.$$
		\end{enumerate}
		\item[$(3)$]  $\min\{|I_{11}|,|I_{12}|\}\ge 2$ and $\big||I_{11}|-|I_{12}|\big|>\varepsilon N$:
		\begin{enumerate}
			\item[$(3a)$] if $\min\{|I_{21}|,|I_{22}|\}=0$, the index of $\boldsymbol{\eta}$ is $\min\{|I_{11}|,|I_{12}|\}\ge 2$;
			\item[$(3b)$] if $\min\{|I_{21}|,|I_{22}|\}=1$, the index of $\boldsymbol{\eta}$ is $\left(\min\{|I_{11}|,|I_{12}|\}+1\right)\ge 3$;
			\item[$(3c)$] if $\min\{|I_{21}|,|I_{22}|\}\ge 2$ and $\big||I_{21}|-|I_{22}|\big|<\varepsilon N$, then the Morse index of $\boldsymbol{\eta}$ is $$\left(\min\{|I_{11}|,|I_{12}|\}+1\right)\ge 3.$$
			\item[$(3d)$] if $\min\{|I_{21}|,|I_{22}|\}\ge 2$ and $\big||I_{21}|-|I_{22}|\big|>\varepsilon N$, then the Morse index of $\boldsymbol{\eta}$ is $$\left(\min\{|I_{11}|,|I_{12}|\}+\min\{|I_{21}|,|I_{22}|\}\right)\ge 4.$$
		\end{enumerate}
	\end{enumerate} 
\end{theorem}
\begin{proof}
	The proof is based on  \cite[Proposition 2.1,Lemma 2.1]{L-Z-Z},  which has computed the spectrum of the Jacobian matrix for   any non-memorized binary pattern \(\boldsymbol{\eta}\) when $\varepsilon$ is {\it any} real number: The total $N$ eigenvalues with their multiplicities are listed  
	\begin{equation} \label{eq:8eig}
	\begin{split}
	&\underbrace{2\big(\frac{-|I_{11}|+|I_{12}|}{N}-\varepsilon\big)}_{|I_{11}|-1,\,\text{if}\,I_{11}\neq \emptyset},\;\; \underbrace{2\big(\frac{|I_{11}|-|I_{12}|}{N}-\varepsilon\big)}_{|I_{12}|-1,\,\text{if}\,I_{12}\neq \emptyset},\;\; \underbrace{2\big( \frac{|I_1|}{N}-\varepsilon \big)}_{1,\,\text{if}\,I_{11}\neq \emptyset\,\text{and}\,I_{12} \neq \emptyset}, \\
	& \underbrace{2\big(\frac{-|I_{21}|+|I_{22}|}{N}-\varepsilon\big)}_{|I_{21}|-1,\,\text{if}\,I_{21}\neq \emptyset},\;\; \underbrace{2\big(\frac{|I_{21}|-|I_{22}|}{N}-\varepsilon\big)}_{|I_{22}|-1,\,\text{if}\,I_{22}\neq \emptyset},\;\; \underbrace{2\big( \frac{|I_2|}{N}-\varepsilon \big)}_{1,\,\text{if}\,I_{21}\neq \emptyset\,\text{and}\,I_{22} \neq \emptyset}, \;\;
	\end{split}     
	\end{equation}
	and  two simple eigenvalues $-2\varepsilon$ and $0$.
	In \eqref{eq:8eig},   the zero multiplicity (when $|I_{11}|=1$ or $|I_{12}|=1$ or $|I_{21}|=1$ or $|I_{22}|=1$) and the empty set (when $I_{11}=\emptyset$ or $I_{12}=\emptyset$ or $I_{21}=\emptyset$ or $I_{22}=\emptyset$)  mean the corresponding eigenvalue does not exist.
	
	Then the main work is to identify  the conditions for each given  number of negative eigenvalues, under the assumption that $0<\varepsilon < \varepsilon^*$. 
	We only show how to verify  Part (1) below and skip  Part (2) and (3).
	
	Part $(1)$: Assume that $\min\{|I_{11}|,|I_{12}|\}=0$. 
	We first examine   the  $|I_1|-1$ eigenvalues  in the first line of \eqref{eq:8eig}.
	The single eigenvalue $2\big(\frac{|I_1|}{N}-\varepsilon\big)$ does not exist. 
	Since  \(I_1\neq\emptyset\), then exactly one of \(|I_{11}|\) and \(|I_{12}|\) is equal to $0$.    
	Without loss of generality, assume that \(|I_{11}| = 0\) and \(|I_{12}|=|I_1|\geq 1\). Then the    eigenvalue   $2\big(\frac{-|I_{11}|+|I_{12}|}{N}-\varepsilon\big)$  does not exist. The  remaining  eigenvalue $2\big(\frac{|I_{11}|-|I_{12}|}{N}-\varepsilon\big)$ either   does not exist   when \(|I_{12}| = 1\), or  is negative for any $\varepsilon\ge 0$. Therefore, $\min\{|I_{11}|,|I_{12}|\}=0$ implies that  all eigenvalues in  the first line of  \eqref{eq:8eig}  either disappear  or are negative, so they do not contribute to the Morse index.
	
	We then  only need to  discuss the   signs of  the eigenvalues in the second line of \eqref{eq:8eig}.
	\begin{equation}\label{eigen}
	\underbrace{2\big(\frac{-|I_{21}|+|I_{22}|}{N}-\varepsilon\big)}_{|I_{21}|-1,\,\text{if}\,I_{21}\neq \emptyset},\;\; ~\underbrace{2\big(\frac{|I_{21}|-|I_{22}|}{N}-\varepsilon\big)}_{|I_{22}|-1,\,\text{if}\,I_{22}\neq \emptyset},\;\; \underbrace{2\big( \frac{|I_2|}{N}-\varepsilon \big)}_{1,\,\text{if}\,I_{21}\neq \emptyset\,\text{and}\,I_{22} \neq \emptyset},   
	\end{equation}

	We remark that 
	$\min\{|I_{11}|,|I_{12}|\}=0$ and $\min\{|I_{21}|,|I_{22}|\}=0$ can not hold simultaneously. Otherwise,    exactly one of $|I_{11}|$ and $|I_{12}|$ must be $0$, and exactly one of $|I_{21}|$ and $|I_{22}|$ must be $0$. Without loss of generality, assume that $|I_{11}| =|I_{21}| = 0$ and \(|I_{12}|,|I_{22}|\geq 1\), then \(\boldsymbol{\eta}= -\boldsymbol{\xi}^1\), which contradicts to \(\boldsymbol{\eta}\notin\{\pm\boldsymbol{\xi}^1,\pm \boldsymbol{\xi}^2\}\). 
	
	Therefore, we have that $\min\{|I_{21}|,|I_{22}|\}\ge 1$
	and the last eigenvalue $2\big(\frac{|I_{2}|}{N}-\varepsilon\big)$ in \eqref{eigen} exists and is  positive because $\varepsilon<\varepsilon^*\le \frac{|I_2|}{N}$, so this le eigenvalue will be counted toward the Morse index.

	$(1a)$ If $\min\{|I_{21}|,|I_{22}|\}=1$, 
	suppose $|I_{22}|=1$  without loss of generality,  
	then the first value in \eqref{eigen} is negative because $\frac{-|I_{21}|+|I_{22}|}{N}=\frac{1-|I_{21}|}{N}\leq 0$.
	The second value in \eqref{eigen} does not exit. 
	Therefore, the Morse index of  $\boldsymbol{\eta}$ is one.
	
	$(1b)$ If $\min\{|I_{21}|,|I_{22}|\}\ge 2$ and $\big||I_{21}|-|I_{22}|\big|<\varepsilon N$, then the first two  numbers  in \eqref{eigen} are negative. So, $\boldsymbol{\eta}$ is of index-$1$.
	
	$(1c)$ If $\min\{|I_{21}|,|I_{22}|\}\ge 2$ and $\big||I_{21}|-|I_{22}|\big|>\varepsilon N$, 
	suppose $|I_{21}|-|I_{22}|>\varepsilon N$ without loss of generality, 
	then the first value in \eqref{eigen} is negative. The second value in \eqref{eigen} is positive due to  $0<\varepsilon<\varepsilon^*$,
	contributing to the Morse index.  The Morse index is then $|I_{22}|$. Similarly, when $|I_{22}|-|I_{21}|>\varepsilon N$, $\boldsymbol{\eta}$ is of index-$|I_{21}|$. Therefore, 
	the Morse index of $\boldsymbol{\eta}$ is  $\min\{|I_{21}|,|I_{22}|\}$.
\end{proof}

\begin{remark}\label{rem2}
	Of particular interest are the index-1 critical points because they lie on the separatrix (basin boundary) between the two stable memorized binary patterns. By the above theorem, 
	if \(0<\varepsilon<\varepsilon^*\), 
	the equivalent 
	conditions for being index-1 is 
	\begin{enumerate}
		\item[either:]~ $\min\{|I_{a1}|,|I_{a2}|\}=0$ and 
		$\min\{|I_{b1}|,|I_{b2}|\}=1$, 
		\item[or:]~   $\min\{|I_{a1}|,|I_{a2}|\}=0$ and 
		$\min\{|I_{b1}|,|I_{b2}|\}\geq 2$ and $\big||I_{b1}|-|I_{b2}|\big|<\varepsilon N$,
	\end{enumerate}
	where    $(a,b)=(1,2)$ or $(2,1)$. 
	These two conditions are easy to check and provide the complete list of {\it all} index-1 binary-pattern critical points, under the  setting $M=2$ and $\varepsilon\in (0,\varepsilon^*)$.
	There are at most $8$ index -$1$ saddle points in case $(i)$:
	\begin{itemize}
		\item $|I_{a1}|=0$, $|I_{a2}|=|I_a|$, $|I_{b1}|=1$, $|I_{b2}|=|I_b|-1$;
		\item $|I_{a1}|=0$, $|I_{a2}|=|I_a|$, $|I_{b1}|=|I_b|-1$, $|I_{b2}|=1$;
		\item $|I_{a1}|=|I_a|$, $|I_{a2}|=0$, $|I_{b1}|=1$, $|I_{b2}|=|I_b|-1$;
		\item $|I_{a1}|=|I_a|$, $|I_{a2}|=0$, $|I_{b1}|=|I_b|-1$, $|I_{b2}|=1$.
	\end{itemize} 
	There are at most $24$ index -$1$ saddle points in case $(ii)$ when $0<\varepsilon<1/N$:
	\begin{itemize}
		\item $|I_{a1}|=0$, $|I_{a2}|=|I_a|$, $|I_{b1}|=|I_{b2}|=|I_b|/2$;
		\item $|I_{a1}|=|I_a|$, $|I_{a2}|=0$, $|I_{b1}|=|I_{b2}|=|I_b|/2$,
	\end{itemize}
	where    $(a,b)=(1,2)$ or $(2,1)$. 
	
	As an example to heuristically understand such index-1 saddle points,  we consider  $|I_{12}|=0$ and $|I_{22}|=1$, then  the oscillators
	in $\boldsymbol{\eta}$  
	matches exactly the memorized pattern wherever the two memorized oscillator agree, and almost matches
	(by only one oscillator misfit for case (i))  one of them wherever the two memories disagree.
\end{remark}

\section{Ternary-pattern  points}\label{sec-mid}

For any binary pattern \(\boldsymbol{\eta}\in\{-1,1\}^N\), the binary-pattern critical point \(\boldsymbol{\varphi}^{*}(\boldsymbol{\eta})\) of system \eqref{mod} defined  by   
$\arccos\boldsymbol{\eta}$ or $\arccos\boldsymbol{-\eta}$
\eqref{Hebbian}  is in $\{0,\pi\}^N$, subject to the translation $\boldsymbol{\varphi}\to \boldsymbol{\varphi}+\delta \mathbbm{1}$ and the sign inversion ($\boldsymbol{\varphi}\to -\boldsymbol{\varphi}$).
Beside the class of oscillators in $\{0,\pi\}^N$ corresponding to the binary patterns,
we shall extend  our discussion 
to a broader class  of   $\{0,\pi/2,\pi\}^N$, consisting of totally $3^N$ points.
We call these points in $\{0,\pi/2,\pi\}^N$  the {\it ternary-pattern points}. One of the motivation to consider this class is  that  the midpoint of any two  binary-pattern  points
associated with 
$\boldsymbol{\eta}$ and $\boldsymbol{\tilde{\eta}}$, 
$\frac12 \boldsymbol{\varphi}^{*}(\boldsymbol{\eta}) + \frac12 \boldsymbol{\varphi}^{*}(\boldsymbol{\tilde{\eta}})$,
takes   possible values of $0,\pi/2,\pi$  for each component. 
Certainly, any ternary-pattern point is the midpoint of two binary-pattern points.
The binary-pattern points
are a special subset of the ternary-pattern points.

We are interested in exploring the same  research questions
for the class of ternary-pattern points:
\begin{enumerate}
	
	\item Are all ternary-pattern points are critical points of \eqref{mod}?
	\item If a  ternary-pattern point is a critical point of system \eqref{mod}, what is its linear stability?
	\item Can we quantify the Morse index for the ternary-pattern critical points?
\end{enumerate}

If the ternary-pattern 
is the midpoint of two of {\it memorized} binary-patterns $\boldsymbol{\xi}^k$,
under the mutual orthogonality condition for memorized binary patterns,
we have the clear answers, depending on whether the second Fourier term in \eqref{mod} exists or not.    We summarizes in the lemma below 
two existing  results \cite[Theorem 2.2]{hölzel2015stability} for $\varepsilon=0$ and \cite[Theorems 4.6-4.7]{zhao2020stability} for $\varepsilon>0$.
\begin{lemma}
	Let $\{\boldsymbol{\xi}^k\}_{k\in [M]}$ be mutually orthogonal memorized binary patterns in system \eqref{mod} 
	and consider the liner combination of any two memories:  $\boldsymbol{\varphi}_u:=\boldsymbol{\varphi}^*(\boldsymbol{\xi}^k)+(\boldsymbol{\varphi}^*(\boldsymbol{\xi}^l)-\boldsymbol{\varphi}^*(\boldsymbol{\xi}^k))u,u\in \mathbb{R}$.
	\begin{itemize}
		\item If $\varepsilon=0$, then $\boldsymbol{\varphi}_u$ is a critical point of \eqref{mod} for each $u\in \mathbb{R}$, 
		thus 
		$\{\boldsymbol{\varphi}^*(\boldsymbol{\xi}^k)\}_{k\in [M]}$ are non-isolated critical points and is only neurally stable.
		\item If $\varepsilon>0$, then 
		$\boldsymbol{\varphi}_u$ is a critical point if and only if $u=\frac{m}{2}$ for $m\in \mathbb{Z}$.
		Furthermore,  $\{\boldsymbol{\varphi}^*(\boldsymbol{\xi}^k)\}_{k\in [M]}$ are asymptotically stable; and  the midpoint $\boldsymbol{\varphi}_{1/2}$ is   unstable.
	\end{itemize}
\end{lemma}

The above results  require the {\it orthogonal} condition for the memorized patterns. For \(\varepsilon=0\), all points on the line connecting any two memorized binary-pattern critical points are critical points of system \eqref{mod}. As \(\varepsilon>0\), these critical points on the lines   disappear, and only the midpoints between memorized binary-pattern critical points remain as critical points, which are unstable. 

But it is not clear whether the midpoint of memories (or more general, in the larger family  of ternary-pattern points) 
is still a critical point, without the orthogonality condition.
We shall   investigate  the three questions mentioned above for the special case $M=2$.

\subsection{Necessary and sufficient condition for the ternary-pattern critical point }

Unlike the binary patterns where $\boldsymbol{\varphi}\in \{0,\pi\}^N$   is always a critical point in our setting of $M=2$, 
the ternary-pattern points $\{0,\pi/2, \pi\}^N$ may not necessarily be a critical point.
We first provide a necessary and sufficient condition for a ternary-pattern point in $\{0,\pi/2,\pi\}^N$ to be a critical point of system \eqref{mod} for $M=2$.
We then show that the midpoint of the two memorized binary patterns, as a special ternary-pattern point, is indeed a critical point.
We  do not assume 
mutual orthogonality of the memorized binary patterns.


\begin{proposition}\textbf{(Ternary-pattern point as critical point: necessity and sufficiency)}\label{prop}
	Let $\{\boldsymbol{\xi}^1,\boldsymbol{\xi}^2\}\subset\{-1,1\}^N$ be two memories in system \eqref{mod}. Define \(I_1,I_2\) as in \eqref{index}.
	For each ternary-pattern point  \({\boldsymbol{\varphi}}=({\varphi}_1,{\varphi}_2,\dots,{\varphi}_N)^{\top}\in\{0,\pi/2,\pi\}^N\), define the following partition of its coordinates:  
	$$X_0:=\{j\in [N]\mid {\varphi}_j=0\}, \quad X_{\pi}:=\{j\in [N]\mid {\varphi}_j=\pi\}, \quad X_{\pi/2}:=\{j\in [N]\mid {\varphi}_j=\pi/2\}.$$
	Let $\varepsilon\in \mathbb{R}$.
	Then $\boldsymbol{\varphi}$ is a critical point of \eqref{mod} 
	if and only if all the following  four conditions  hold for $X_0, X_{\pi}, X_{\pi/2}$:
	\begin{align}\label{X123}
	\begin{split}
	&  \sum\nolimits_{j\in I_1\cap X_{\pi/2}} \xi^1_j=0,\quad\quad \quad\quad \quad\quad  \text{ if }\;   I_1\cap (X_0\cup X_{\pi}) \neq \emptyset;\\
	& \sum\nolimits_{j\in I_2\cap X_{\pi/2}} \xi^1_j=0,\quad \quad\quad  \quad\quad\quad \text{ if }\; I_2\cap (X_0\cup X_{\pi}) \neq \emptyset;\\
	& \sum\nolimits_{j\in I_1\cap X_0} \xi^1_j=\sum\nolimits_{j\in I_1\cap X_{\pi}} \xi^1_j,\quad \text{ if }\;  I_1\cap X_{\pi/2} \neq \emptyset;\\
	&\sum\nolimits_{j\in I_2\cap X_0} \xi^1_j=\sum\nolimits_{j\in I_2\cap X_{\pi}} \xi^1_j,\quad \text{ if }\;  I_2\cap X_{\pi/2} \neq \emptyset.
	\end{split}
	\end{align}
\end{proposition}
\begin{proof}
	Note that the second-order term in \eqref{mod} for ternary-patter points vanish since 
	$\sum_{j=1}^{N}\sin 2({\varphi}_j-{\varphi}_i)\equiv 0$ for all $\varphi_i,\varphi_j \in \{0,\pi/2,\pi\}$.
	We only need study $\sum_{j=1}^{N}C_{ij}\sin({\varphi}_j-{\varphi}_i)$.

	For $i\in I_1$ and $j\in [N]$, we have that $\xi^1_i=\xi^2_i$ and $C_{ij}=\xi^1_i (\xi^1_j+\xi^2_j)$. Then $C_{ij}=2\xi^1_i\xi^1_j$ for $j\in I_1$; $C_{ij}=0$ for $j\in I_2$.
	We obtain that for $i\in I_1$,
	\begin{align*}
	&\quad \sum\nolimits_{j=1}^{N}C_{ij}\sin({\varphi}_j-{\varphi}_i)    =2\xi^1_i\sum\nolimits_{j\in I_1}\xi^1_j \sin({\varphi}_j-{\varphi}_i)\\
	& =2\xi^1_i \left[-\sum\nolimits_{j\in I_1\cap X_0}\xi^1_j\sin{\varphi}_i+\sum\nolimits_{j\in I_1\cap X_{\pi}}\xi^1_j\sin{\varphi}_i+\sum\nolimits_{j\in I_1\cap X_{\pi/2}}\xi^1_j\cos{\varphi}_i  \right]\\
	& =2\xi^1_i\left[\sin{\varphi}_i\left( \sum\nolimits_{j\in I_1\cap X_{\pi}}\xi^1_j-\sum\nolimits_{j\in I_1\cap X_{0}}\xi^1_j \right) +\cos{\varphi}_i \sum\nolimits_{j\in I_1\cap X_{\pi/2}}\xi^1_j \right], 
	\end{align*}
	which implies 
	\begin{equation}\label{iff1}
	\sum\nolimits_{j=1}^{N}C_{ij}\sin({\varphi}_j-{\varphi}_i)=\begin{cases}2\xi^1_i \sum_{j\in I_1\cap X_{\pi/2}}\xi^1_j, &\text{ if } i\in I_1\cap X_0;\\ -2\xi^1_i \sum_{j\in I_1\cap X_{\pi/2}}\xi^1_j, &\text{ if } i\in I_1\cap X_{\pi};\\ 2\xi^1_i \left( \sum_{j\in I_1\cap X_{\pi}}\xi^1_j-\sum_{j\in I_1\cap X_0}\xi^1_j \right), &\text{ if }i\in I_1\cap X_{\pi/2}. \end{cases}
	\end{equation}
	
	    Similarly, 
	for $i\in I_2$, we have that 
	\begin{equation}\label{iff2}
	\sum\nolimits_{j=1}^{N}C_{ij}\sin({\varphi}_j-{\varphi}_i)=\begin{cases}2\xi^1_i \sum_{j\in I_2\cap X_{\pi/2}}\xi^1_j, &\text{ if } i\in I_2\cap X_0;\\ -2\xi^1_i \sum_{j\in I_2\cap X_{\pi/2}}\xi^1_j, &\text{ if } i\in I_2\cap X_{\pi};\\ 2\xi^1_i \left( \sum_{j\in I_2\cap X_{\pi}}\xi^1_j-\sum_{j\in I_2\cap X_0}\xi^1_j \right), &\text{ if } i\in I_2\cap X_{\pi/2}. \end{cases}
	\end{equation}
	The desired results are then  obtained.
\end{proof}

It is in general  difficult to verify the conditions in Proposition \ref{prop}. But there exists one trivial family of ternary-pattern points which are always the critical points, as shown below.

\begin{corollary}\label{cor3}
Let $\{\boldsymbol{\xi}^1,\boldsymbol{\xi}^2\}\subset\{-1,1\}^N$ be two memories in system \eqref{mod} and the sets $\{I_{i}\}_{i=1,2}$ are defined in \eqref{index}. Then, each  ternary-pattern point \(\widehat{\boldsymbol{\varphi}}=(\widehat{\varphi}_1,\widehat{\varphi}_2,\dots,\widehat{\varphi}_N)^{\top}\in \{0,\pi/2,\pi\}^N\)   satisfying 
\begin{equation}\label{mid1}
\widehat{\varphi}_i=\begin{cases}0 \mbox{ or } \pi, &\text{ if } i\in I_1;\\ \pi/2, &\text{ if } i\in I_2, \end{cases}\end{equation}
or   \begin{equation}\label{mid2}
\widehat{\varphi}_i=\begin{cases}\pi/2, &\text{ if }i\in I_1;\\ 0 \mbox{ or } \pi, &\text{ if } i\in I_2,  \end{cases}
\end{equation}
is a critical point of system \eqref{mod} for any $\varepsilon$. 

In particular,  the midpoints, \((\boldsymbol{\varphi}^{*}(\boldsymbol{\xi}^1)+\boldsymbol{\varphi}^{*}(\boldsymbol{\pm \xi}^2))/2\), of  $\boldsymbol{\varphi}^{*}(\boldsymbol{\xi}^1)$ and $\boldsymbol{\varphi}^{*}(\pm\boldsymbol{\xi}^2)$, with $\boldsymbol{\varphi}^*=\arccos$, satisfy the condition  \eqref{mid1} or \eqref{mid2}, respectively, 
So, they are critical points of system \eqref{mod}.
\end{corollary}

\begin{proof}
\eqref{mid1} means that  $I_2=X_{\pi/2}$ and  $I_1=X_{0}\cup X_{\pi}$. So,  $I_1\cap X_{\pi/2}=I_2\cap (X_0\cup X_{\pi})=\emptyset$. So four conditions  in 
\eqref{X123} are null and fulfilled.  By Proposition \ref{prop},  \eqref{mid1} is a critical point.
The proof for    \eqref{mid2}   is  the same by replacing $\boldsymbol{\xi}^2$ with $-\boldsymbol{\xi}^2$ (or equivalently,
by replacing $\boldsymbol{\xi}^1$ with $-\boldsymbol{\xi}^1$) .  
The verification for the midpoint case is straightforward. 
\end{proof}

For convenience, the ternary-pattern critical points that satisfy the condition \eqref{mid1} or \eqref{mid2}  are called {\bf  memory-compatible ternary-pattern critical points} to highlight the dependence on the memorized binary-patterns via $I_1$ and $I_2$ (recall \eqref{index}).  
There are $2^{|I_1|}+2^{|I_2|}$
such  ternary-pattern critical points,  including the images due to the invariance $\boldsymbol{\varphi}\to \pi -\boldsymbol{\varphi}$.
We next study the spectrum of these 
memory-compatible ternary-pattern critical points in order to identify the possible index-1 saddle points. 
Since \eqref{mid1} and \eqref{mid2} are 
similar,   we only   study  \eqref{mid1}.

\subsection{Spectrum  and stability of memory-compatible ternary-pattern points}\label{subsec4.2}

We calculate  the spectrum of the Jacobian matrix $\boldsymbol{\mathcal{J}}_{\widehat{\boldsymbol{\varphi}}}$ at $\widehat{\boldsymbol{\varphi}}$ satisfying \eqref{mid1} in this subsection.
First, we need  the following  definitions of index partitions for $I_1=S_1\cup S_2$ and $I_2=S_3\cup S_4$ (recall \eqref{mid1} and \eqref{index}): 
\begin{align}\label{S1234}
\begin{split}
& S_1:=\left\{i\in [N] \mid \xi^1_i=\xi^2_i=1 \quad \text{and} \quad \widehat{\varphi}_i=0   \right\}\bigcup \left\{ i\in [N] \mid \xi^1_i=\xi^2_i=-1 \quad \text{and} \quad \widehat{\varphi}_i=\pi \right\},\\
& S_2:=\left\{i\in [N] \mid \xi^1_i=\xi^2_i=1 \quad \text{and} \quad \widehat{\varphi}_i=\pi   \right\}\bigcup \left\{ i\in [N] \mid \xi^1_i=\xi^2_i=-1 \quad \text{and} \quad \widehat{\varphi}_i=0 \right\},\\
& S_3:=\left\{i\in [N] \mid \xi^1_i=-\xi^2_i=1 \quad \text{and} \quad \widehat{\varphi}_i=\pi/2\right\},\\
&S_4:=\left\{ i\in [N] \mid \xi^1_i=-\xi^2_i=-1 \quad \text{and} \quad \widehat{\varphi}_i=\pi/2\right\}.
\end{split}
\end{align}

For the exploration of eigenvectors, 
we define the linear space $W(S)$ in $\mathbb{R}^N$ 
for   any nonempty set $S\subset [N]$  as follows:
\[ W(S):=\left\{ (x_1,x_2,\dots, x_N)^{\top}\in \mathbb{R}^N \mid x_j=0\; \text{for}\; j\notin S, \; \text{and}\; \sum\nolimits_{j\in S}x_j=0 \right\} .\]
The dimension of $W(S)$ is $|S|-1.$  For  $\{S_i\}_{i=1}^4$ defined above, $W(S_i)$ are then defined accordingly.
In addition, we define $W[\mathbbm{1}]:=\{c \mathbbm{1}\mid c \in\mathbb{R}\}$,
where $\mathbbm{1}$ represents the vector whose components are all $1$.  

Given any two non-empty sets $S, \tilde{S}\subset [N]$ with $S\cap \tilde{S}=\emptyset$, define the vector  $\boldsymbol{x}^{*}
(S;\tilde{S})=(x^*_1,x^*_2,\dots, x^*_N)^{\top}$ satisfying $\sum_{i=1}^{N}x^*_i=0$ as follows  
\[ 	x^*_i:=\begin{cases}|\tilde{S}|, &\text{ if } i\in S;\\ -|S|, & \text{ if } i\in \tilde{S}; \\ 0, & \text{otherwise}. \end{cases} \]
and   the one dimensional subspace 
\[W(S; \tilde{S}):=\left\{c\boldsymbol{x}^{*}(S;\tilde{S}) \mid c\in \mathbb{R}\right\}
\subset  \mathbb{R}^N,\]
which will provide  the subspaces  $W(S_1; S_2), W(S_3; S_4)$ and $W(I_1; I_2)$ respectively. 

Before proving Theorem \ref{thm-mid}, we first show a useful lemma.
\begin{lemma}
	Let $\{\boldsymbol{\xi}^1,\boldsymbol{\xi}^2\}\subset\{-1,1\}^N$ be two memories in system  \eqref{mod}. Let \(\widehat{\boldsymbol{\varphi}}=(\widehat{\varphi}_1,\widehat{\varphi}_2,\dots,\widehat{\varphi}_N)^{\top}\) be the memory-compatible ternary-pattern point  defined by  \eqref{mid1}. Then
	\begin{equation}\label{cij}
	C_{ij}\cos(\widehat{\varphi}_j-\widehat{\varphi}_i)=\begin{cases}(\xi^1_i+\xi^2_i)\cos\widehat{\varphi}_i, &\text{ if }j\in S_1;\\ -(\xi^1_i+\xi^2_i)\cos\widehat{\varphi}_i, &\text{ if }j\in S_2; \\ (\xi^1_i-\xi^2_i)\sin\widehat{\varphi}_i, &\text{ if }j\in S_3; \\ -(\xi^1_i-\xi^2_i)\sin\widehat{\varphi}_i, &\text{ if }j\in S_4, \end{cases} \quad \text{for each}\; i,j\in [N],
	\end{equation} and
	\begin{equation}\label{cos2}
	\cos 2(\widehat{\varphi}_j-\widehat{\varphi}_i)=\begin{cases}\cos 2\widehat{\varphi}_i, &\text{ if }j\in I_1;\\ -\cos 2\widehat{\varphi}_i, &\text{ if }j\in I_2, \end{cases} \quad \text{for each}\; i,j\in [N].
	\end{equation}
	For the Jacobian matrix $\boldsymbol{\mathcal{J}}_{\widehat{\boldsymbol{\varphi}}}$ and for any real vector $\boldsymbol{x}=(x_1,x_2,\dots,x_N)^{\top}$ with $\sum_{j=1}^{N}x_j=0$, 
	\begin{align}\label{Jaco3}
	(\boldsymbol{\mathcal{J}}_{\widehat{\boldsymbol{\varphi}}}\boldsymbol{x})_i&=\frac{1}{N}\left[ (\xi^1_i+\xi^2_i)\cos\widehat{\varphi}_i \left(\sum\nolimits_{j\in S_1}x_j-\sum\nolimits_{j\in S_2}x_j\right)+(\xi^1_i+\xi^2_i)\cos\widehat{\varphi}_i (|S_2|-|S_1|)x_i \right.\nonumber \\
	&\left.\quad +(\xi^1_i-\xi^2_i)\sin\widehat{\varphi}_i \left(\sum\nolimits_{j\in S_3}x_j-\sum\nolimits_{j\in S_4}x_j\right)+(\xi^1_i-\xi^2_i)\sin\widehat{\varphi}_i (|S_4|-|S_3|)x_i \right]\nonumber\\
	&\quad +\frac{2\varepsilon}{N}\cos2\widehat{\varphi}_i\left[\sum\nolimits_{j\in I_1}x_j- \sum\nolimits_{j\in I_2}x_j+(|I_2|-|I_1|)x_i\right],
	\end{align}
	where
	$(\boldsymbol{\mathcal{J}}_{\widehat{\boldsymbol{\varphi}}}\boldsymbol{x})_i$ represents the $i$-th component of $\boldsymbol{\mathcal{J}}_{\widehat{\boldsymbol{\varphi}}}\boldsymbol{x}$ for $i\in [N]$. 
\end{lemma}
\begin{proof}
	Regarding \eqref{cij}, we see from \eqref{mid1} and the definitions of $\{S_i\}_{i=1}^{4}$ and $I_1,I_2$ that for $j\in S_1$, $\xi^1_j=\xi^2_j$, $\xi^1_j\cos\widehat{\varphi}_j=1$ and $\sin\widehat{\varphi}_j=0$. Then
	\[ 	C_{ij}\cos(\widehat{\varphi}_j-\widehat{\varphi}_i)=(\xi^1_i\xi^1_j+\xi^2_i\xi^2_j)\cos(\widehat{\varphi}_j-\widehat{\varphi}_i)=(\xi^1_i+\xi^2_i)\xi^1_j\cos(\widehat{\varphi}_j-\widehat{\varphi}_i)=(\xi^1_i+\xi^2_i)\cos\widehat{\varphi}_i.  \]
	For $j\in S_2, S_3, S_4$, the desired results can be obtained in a similar way. Equation \eqref{cos2} is obvious.
	
	    Regarding \eqref{Jaco3}, by $\eqref{Jaco1}$-$\eqref{Jaco2}$ and \(\sum_{j=1}^{N}x_j=0\), we observe that $\sum_{j=1,j\neq i}x_j=-x_i$ and 
	\begin{align*}
	(\boldsymbol{\mathcal{J}}_{\widehat{\boldsymbol{\varphi}}}\boldsymbol{x})_i&=\sum\nolimits_{\substack{j=1 \\ j \neq i}}^{N}\left(\frac{1}{N}C_{ij}\cos(\widehat{\varphi}_j-\widehat{\varphi}_i)+\frac{2\varepsilon}{N}\cos 2(\widehat{\varphi}_j-\widehat{\varphi}_i) \right)x_j\\
	&\quad -\left[ \sum\nolimits_{\substack{j=1 \\ j \neq i}}^{N} \frac{1}{N}C_{ij}\cos(\widehat{\varphi}_j-\widehat{\varphi}_i)+\frac{2\varepsilon}{N}\sum\nolimits_{\substack{j=1 \\ j \neq i}}^{N}\cos 2(\widehat{\varphi}_j-\widehat{\varphi}_i)\right]x_i\\
	&=\frac{1}{N}\sum\nolimits_{j=1}^{N}C_{ij}\cos(\widehat{\varphi}_j-\widehat{\varphi}_i)(x_j-x_i)+\frac{2\varepsilon}{N}\sum\nolimits_{\substack{j=1}}^{N}\cos 2(\widehat{\varphi}_j-\widehat{\varphi}_i)(x_j-x_i),\;\; \text{for}\; i\in [N].
	\end{align*}
	Using equations \eqref{cij}-\eqref{cos2}, we obtain that
	\begin{align*}
	(\boldsymbol{\mathcal{J}}_{\widehat{\boldsymbol{\varphi}}}\boldsymbol{x})_i&=\frac{1}{N}\left[\left(\sum\nolimits_{j\in S_1}+\sum\nolimits_{j\in S_2}+\sum\nolimits_{j\in S_3}+\sum\nolimits_{j\in S_4} \right)C_{ij}\cos(\widehat{\varphi}_j-\widehat{\varphi}_i)(x_j-x_i) \right]\\
	&\quad +\frac{2\varepsilon}{N}\left[\left(\sum\nolimits_{j\in I_1}-\sum\nolimits_{j\in I_2}\right)(x_j-x_i)\cos 2\widehat{\varphi}_i  \right]       \\
	&=\frac{1}{N}\left[(\xi^1_i+\xi^2_i)\cos\widehat{\varphi}_i\sum\nolimits_{j\in S_1}(x_j-x_i)- (\xi^1_i+\xi^2_i)\cos\widehat{\varphi}_i\sum\nolimits_{j\in S_2}(x_j-x_i)\right.\\
	&\quad+ \left. (\xi^1_i-\xi^2_i)\sin\widehat{\varphi}_i\sum\nolimits_{j\in S_3}(x_j-x_i)- (\xi^1_i-\xi^2_i)\sin\widehat{\varphi}_i\sum\nolimits_{j\in S_4}(x_j-x_i)
	\right]\\
	&\quad +\frac{2\varepsilon}{N}\cos 2\widehat{\varphi}_i \left[\sum\nolimits_{j\in I_1}x_j-\sum\nolimits_{j\in I_2} x_j -|I_1|x_i+|I_2|x_i\right] \\
	&=\frac{1}{N}\left[(\xi^1_i+\xi^2_i)\cos\widehat{\varphi}_i\left( \sum\nolimits_{j\in S_1}x_j-\sum\nolimits_{j\in S_2}x_j\right)+(\xi^1_i+\xi^2_i)\cos\widehat{\varphi}_i\left( |S_2|-|S_1|\right)x_i\right.\\
	&\quad +\left. (\xi^1_i-\xi^2_i)\sin\widehat{\varphi}_i\left( \sum\nolimits_{j\in S_3}x_j-\sum\nolimits_{j\in S_4}x_j\right)+(\xi^1_i-\xi^2_i)\sin\widehat{\varphi}_i\left( |S_4|-|S_3|\right)x_i \right]\\
	&\quad +\frac{2\varepsilon}{N}\cos2\widehat{\varphi}_i\left[\sum\nolimits_{j\in I_1}x_j- \sum\nolimits_{j\in I_2}x_j+(|I_2|-|I_1|)x_i\right],\;\; \text{for}\; i\in [N].
	\end{align*}
	The proof of equation \eqref{Jaco3} has been accomplished. 
\end{proof}

\begin{theorem}\textbf{(Spectrum of memory-compatible ternary-pattern points)}\label{thm-mid}
Let $\{\boldsymbol{\xi}^1,\boldsymbol{\xi}^2\}\subset\{-1,1\}^N$ be two memories in system  \eqref{mod} and $\varepsilon\in\mathbb{R}$. Let \(\widehat{\boldsymbol{\varphi}}=(\widehat{\varphi}_1,\widehat{\varphi}_2,\dots,\widehat{\varphi}_N)^{\top}\) 
be a memory-compatible ternary-pattern point satisfying  \eqref{mid1}. Then for the Jacobian matrix $\boldsymbol{\mathcal{J}}_{\widehat{\boldsymbol{\varphi}}}$,
\begin{enumerate}
	\item[$(1)$]  $0$ is a single eigenvalue with the eigenspace $W[\mathbbm{1}]$;
	\item[$(2)$] $2\varepsilon$ is a single eigenvalue with the eigenspace $W(I_1; I_2)$;
	\item[$(3)$] if $S_1\neq \emptyset$, then $\frac{2(-|S_1|+|S_2|)}{N}+\frac{2\varepsilon(|I_2|-|I_1|)}{N}$ is an  eigenvalue with the multiplicity $|S_1|-1$, corresponding to the eigenspace $W(S_1)$;
	\item[$(4)$] if $S_2\neq \emptyset$, then $\frac{2(|S_1|-|S_2|)}{N}+\frac{2\varepsilon(|I_2|-|I_1|)}{N}$ is an  eigenvalue with the multiplicity $|S_2|-1$, corresponding to the eigenspace $W(S_2)$;
	\item[$(5)$] if $S_1\neq \emptyset$ and $S_2 \neq \emptyset$, then $\frac{2|I_1|}{N}+\frac{2\varepsilon(|I_2|-|I_1|)}{N}$ is a single eigenvalue with the eigenspace $W(S_1; S_2)$;
	\item[$(6)$] if $S_3\neq \emptyset$, then $\frac{2(-|S_3|+|S_4|)}{N}+\frac{2\varepsilon(|I_1|-|I_2|)}{N}$ is an eigenvalue with the multiplicity $|S_3|-1$, corresponding to the eigenspace $W(S_3)$;
	\item[$(7)$] if $S_4\neq \emptyset$, then  $\frac{2(|S_3|-|S_4|)}{N}+\frac{2\varepsilon(|I_1|-|I_2|)}{N}$ is an eigenvalue with the multiplicity $|S_4|-1$, corresponding to the eigenspace $W(S_4)$;
	\item[$(8)$] if $S_3\neq \emptyset$ and $S_4 \neq \emptyset$, then $\frac{2|I_2|}{N}+\frac{2\varepsilon(|I_1|-|I_2|)}{N}$  is a single eigenvalue with the eigenspace $W(S_3; S_4)$.
\end{enumerate}
In $(3)$-$(4)$ and $(6)$-$(7)$,   the zero multiplicity (when $|S_{1}|=1$ or $|S_{2}|=1$ or $|S_{3}|=1$ or $|S_{4}|=1$) and the empty set (when $S_{1}=\emptyset$ or $S_{2}=\emptyset$ or $S_{3}=\emptyset$ or $S_{4}=\emptyset$)  mean the corresponding eigenvalue does not exist.

Obviously, the memory-compatible ternary-pattern point  satisfying  \eqref{mid1} must be unstable if $\varepsilon>0$ because its Jacobian has a positive eigenvalue $2\varepsilon$.
\end{theorem}
\begin{proof}
By $\sum_{i=1}^{4}|S_i|=|I_1|+|I_2|=N$ and $\dim W[\mathbbm{1}]=\dim W(I_1; I_2)=\dim W(S_1; S_2)=\dim W(S_3; S_4)=1,$ we find that the dimensions
of those subspaces stated in $(1)$-$(8)$ sum to $N$. In fact, the sum of dimensions of spaces in
$(3)$-$(5)$ is $|I_1|-1$, and the sum of dimensions of spaces in $(6)$-$(8)$ is $|I_2|-1$. Therefore, to
verify the results in this theorem, it suffices to show that any nonzero vector in each space is
an eigenvector of the corresponding eigenvalue.

$(1)$ The relation $\boldsymbol{\mathcal{J}}_{\widehat{\boldsymbol{\varphi}}}\mathbbm{1}=0\mathbbm{1}$ is obvious, due to the global phase shift invariance. 

$(2)$ It suffices to show $\boldsymbol{\mathcal{J}}_{\widehat{\boldsymbol{\varphi}}}\boldsymbol{x}^*=2\varepsilon \boldsymbol{x}^*$ where $\boldsymbol{x}^*=\boldsymbol{x}^*(I_1; I_2)\in W(I_1;I_2)$. By \eqref{Jaco3}, the desired result is equivalent to 
\begin{align}
& (\xi^1_i+\xi^2_i)\cos\widehat{\varphi}_i \left(\sum\nolimits_{j\in S_1}x_j^*-\sum\nolimits_{j\in S_2}x_j^*\right)+(\xi^1_i+\xi^2_i)\cos\widehat{\varphi}_i (|S_2|-|S_1|)x_i^*\label{equi1} \\
&\quad +(\xi^1_i-\xi^2_i)\sin\widehat{\varphi}_i \left(\sum\nolimits_{j\in S_3}x_j^*-\sum\nolimits_{j\in S_4}x_j^*\right)+(\xi^1_i-\xi^2_i)\sin\widehat{\varphi}_i (|S_4|-|S_3|)x_i^*\equiv 0,\; \text{for}\;\; i\in [N],\nonumber
\end{align} and 
\begin{align}\label{equi2}
\frac{2\varepsilon}{N}\cos 2\widehat{\varphi}_i\left[ \sum\nolimits_{j\in I_1}x^*_j-\sum\nolimits_{j\in I_2}x^*_j+(|I_2|-|I_1|)x^*_i \right]=2\varepsilon x^*_i,\; \text{for}\;\; i\in [N].
\end{align}
In fact, based on the definitions of $\boldsymbol{x}^*(I_1; I_2)$ and $\{S_i\}_{i=1}^{4}$, if $i\in I_1$, then we obtain that
\[\sum\nolimits_{j\in S_1}x_j^*=|I_2| |S_1|, \quad \sum\nolimits_{j\in S_2}x_j^*=|I_2| |S_2|, \quad x^*_i=|I_2|\quad \text{and}\quad \xi^1_i=\xi^2_i. \]
Substituting the above into equation \eqref{equi1}, the equation \eqref{equi1} holds.
Similarly, it can be shown that equation \eqref{equi1} holds when $i\in I_2$.
Regarding \eqref{equi2}, if $i\in I_1$, then $\cos 2\widehat{\varphi}_i=1$, $x^*_i=|I_2|$, and $\sum\nolimits_{j\in I_1}x^*_j-\sum\nolimits_{j\in I_2}x^*_j=2|I_1||I_2|$. Hence equation \eqref{equi2} holds for $i\in I_1$. With a similar proof, we see that equation \eqref{equi2} also holds for $i\in I_2$.

	$(3)$ For any $\boldsymbol{x}=(x_1,x_2,\dots,x_N)^{\top}\in W(S_1)$, we need to prove that 
\[ (\boldsymbol{\mathcal{J}}_{\widehat{\boldsymbol{\varphi}}}\boldsymbol{x})_i=0\; \;\text{for}\; i\notin S_1\quad \text{and}\quad  (\boldsymbol{\mathcal{J}}_{\widehat{\boldsymbol{\varphi}}}\boldsymbol{x})_i=\left[ \frac{2(-|S_1|+|S_2|)}{N}+\frac{2\varepsilon(|I_2|-|I_1|)}{N}\right]x_i\; \;\text{for}\; i\in S_1.  \]
Using the definition of $W(S_1)$, we obtain that $\sum_{j\in S_1}x_j=0$ and $x_j=0$ for $j\notin S_1$, that is,
\[ \sum\nolimits_{j\in S_1}x_j=\sum\nolimits_{j\in S_2}x_j=\sum\nolimits_{j\in S_3}x_j=\sum\nolimits_{j\in S_4}x_j=0\quad \text{and}\quad \sum\nolimits_{j\in I_1}x_j-\sum\nolimits_{j\in I_2}x_j=0. \]
From equation \eqref{Jaco3}, we see that for $i\in [N]$,
\begin{align*}
(\boldsymbol{\mathcal{J}}_{\widehat{\boldsymbol{\varphi}}}\boldsymbol{x})_i&=\frac{1}{N}\left[(\xi^1_i+\xi^2_i)\cos\widehat{\varphi}_i (|S_2|-|S_1|)x_i+(\xi^1_i-\xi^2_i)\sin\widehat{\varphi}_i (|S_4|-|S_3|)x_i \right]\\
&\quad +\frac{2\varepsilon}{N} (|I_2|-|I_1|)\cos 2\widehat{\varphi}_i\ x_i.
\end{align*}

If $i\notin S_1$, then by $x_i=0$ we have that $(\boldsymbol{\mathcal{J}}_{\widehat{\boldsymbol{\varphi}}}\boldsymbol{x})_i=0$. If $i\in S_1$, then $\xi^1_i=\xi^2_i$ and $\xi^1_i\cos\widehat{\varphi}_i=1$ and $\cos 2\widehat{\varphi}_i=1$. So, 
$(\boldsymbol{\mathcal{J}}_{\widehat{\boldsymbol{\varphi}}}\boldsymbol{x})_i=\frac{2(|S_2|-|S_1|)x_i}{N}+\frac{2\varepsilon (|I_2|-|I_1|)}{N} x_i. $

$(4)$ The proof is similar to that of $(3)$, and thus it is omitted here.

$(5)$ It suffices to show $(\boldsymbol{\mathcal{J}}_{\widehat{\boldsymbol{\varphi}}}\boldsymbol{x}^*)_i=\left[\frac{2|I_1|}{N}+\frac{2\varepsilon(|I_2|-|I_1|)}{N} \right]x^*_i$ for $i\in [N]$ where $\boldsymbol{x}^*=\boldsymbol{x}^*(S_1; S_2)\in W(S_1;S_2)$. This is equivalent to 
\begin{align}
(\boldsymbol{\mathcal{J}}_{\widehat{\boldsymbol{\varphi}}}\boldsymbol{x}^*)_i&=\left[\frac{2|I_1|}{N}+\frac{2\varepsilon(|I_2|-|I_1|)}{N} \right]|S_2|, \;\; \text{for}\;\; i\in S_1, \label{equ11}\\
(\boldsymbol{\mathcal{J}}_{\widehat{\boldsymbol{\varphi}}}\boldsymbol{x}^*)_i&=-\left[\frac{2|I_1|}{N}+\frac{2\varepsilon(|I_2|-|I_1|)}{N} \right]|S_1|, \;\; \text{for}\;\; i\in S_2,  \label{equ12}\\
(\boldsymbol{\mathcal{J}}_{\widehat{\boldsymbol{\varphi}}}\boldsymbol{x}^*)_i&=0,\;\; \text{for}\;\; i\in I_2.\label{equ13}
\end{align}
Note that $S_1\cup S_2=I_1$ and $I_1\cup I_2=[N]$. Using the definition of $\boldsymbol{x}^*(S_1; S_2)$, we have that
\[\sum_{j\in S_1}x^*_j=|S_1||S_2|,\quad \sum_{j\in S_2}x^*_j=-|S_1||S_2|,\quad \sum_{j\in S_3}x^*_j=\sum_{j\in S_4}x^*_j=0 \quad \text{and}\quad \sum_{j\in I_1}x^*_j=\sum_{j\in I_2}x^*_j=0, \]
and $x^*_j=0$ for $j\in I_2$. Then by equation \eqref{Jaco3}, we obtain that
\begin{align*}
(\boldsymbol{\mathcal{J}}_{\widehat{\boldsymbol{\varphi}}}\boldsymbol{x}^*)_i&=\frac{1}{N}\left[ 2(\xi^1_i+\xi^2_i)\cos\widehat{\varphi}_i |S_1||S_2|+(\xi^1_i+\xi^2_i)\cos\widehat{\varphi}_i (|S_2|-|S_1|)x_i^* \right.\\
&\left.\quad +(\xi^1_i-\xi^2_i)\sin\widehat{\varphi}_i (|S_4|-|S_3|)x_i^* \right]+\frac{2\varepsilon}{N} (|I_2|-|I_1|)\cos 2\widehat{\varphi}_i\ x_i^*, \;\; \text{for}\; i\in [N].
\end{align*}

If $i\in S_1$, then $\xi^1_i=\xi^2_i$ and $\xi^1_i\cos\widehat{\varphi}_i=1$ and $x^*_i=|S_2|$ and $\cos 2\widehat{\varphi}_i=1$. Equation \eqref{equ11} is obtained. If $i\in S_2$, then $\xi^1_i=\xi^2_i$ and $\xi^1_i\cos\widehat{\varphi}_i=-1$ and $x^*_i=-|S_1|$ and $\cos 2\widehat{\varphi}_i=1$. Equation \eqref{equ12} holds.  If $i\in I_2$, then $\xi^1_i=-\xi^2_i$ and $x^*_i=0$ and $\cos 2\widehat{\varphi}_i=0$, which indicates equation \eqref{equ13}.

$(6)$-$(8)$ Analogous to the analysis of $(3)$-$(5)$, the proof is omitted here.

	\end{proof}

For the sake of simplicity, we summarize the eigenvalues  in Theorem \ref{thm-mid}:
\begin{align}
&\underbrace{\frac{2(-|S_1|+|S_2|)}{N}+\frac{2\varepsilon(|I_2|-|I_1|)}{N}}_{|S_{1}|-1,\,\text{if}\,S_{1}\neq \emptyset},\quad \underbrace{\frac{2(|S_1|-|S_2|)}{N}+\frac{2\varepsilon(|I_2|-|I_1|)}{N}}_{|S_{2}|-1,\,\text{if}\,S_{2}\neq \emptyset},\quad \underbrace{\frac{2|I_1|}{N}+\frac{2\varepsilon(|I_2|-|I_1|)}{N}}_{1,\,\text{if}\,S_{1}\neq \emptyset\,\text{and}\,S_{2} \neq \emptyset}, \label{mid-eig1}\\
& \underbrace{\frac{2(-|S_3|+|S_4|)}{N}+\frac{2\varepsilon(|I_1|-|I_2|)}{N}}_{|S_{3}|-1,\,\text{if}\,S_{3}\neq \emptyset},\quad \underbrace{\frac{2(|S_3|-|S_4|)}{N}+\frac{2\varepsilon(|I_1|-|I_2|)}{N}}_{|S_{4}|-1,\,\text{if}\,S_{4}\neq \emptyset},\quad \underbrace{\frac{2|I_2|}{N}+\frac{2\varepsilon(|I_1|-|I_2|)}{N}}_{1,\,\text{if}\,S_{3}\neq \emptyset\,\text{and}\,S_{4} \neq \emptyset},\label{mid-eig2}\\
& 0,\;\; 2\varepsilon.\label{mid-eig3}
\end{align}

Next, we show that  the
number of index-$1$ saddle points
among the total $2^{|I_1|}$ memory-compatible ternary-pattern points is in fact only  four  for a small $\varphi$. 
The following theorem shows the sufficient and necessary condition 
for  a memory-compatible ternary-pattern point to be an index-1 saddle point.

\begin{theorem}\label{Morthm1}
Under the same assumption in Theorem \ref{thm-mid}, 
if $I_1,I_2$ and $\varepsilon$ satisfy one of the following,
\begin{itemize}
	\item  $|I_1|=|I_2|$
	(i.e., $\boldsymbol{\xi}^1$ and $\boldsymbol{\xi}^1$ are orthogonal) and $0<\varepsilon<+\infty$, 
	\item  $|I_1|\neq |I_2|$ and  $0<\varepsilon<\widehat{\varepsilon}:=
	\frac{\min\{|I_1|,|I_2|\}}
	{\big||I_1|-|I_2|\big|},$ 
\end{itemize} 
then the Morse index of a memory-compatible ternary pattern critical point $\widehat{\boldsymbol{\varphi}}$  satisfying \eqref{mid1}, is $1$,  if and only if 
\begin{equation} \label{eq:min4S}
\min\{|S_1|,|S_2|\}=\min\{|S_3|,|S_4|\}=0.
\end{equation}
\end{theorem}
\begin{proof}
$(1)$ For $|I_1|=|I_2|$ and $0<\varepsilon<+\infty$, \eqref{mid-eig1}-\eqref{mid-eig3} are reduced to
\begin{align}\label{mid-eig4}
\begin{split}
&\underbrace{\frac{2(-|S_1|+|S_2|)}{N}}_{|S_{1}|-1,\,\text{if}\,S_{1}\neq \emptyset},\quad \underbrace{\frac{2(|S_1|-|S_2|)}{N}}_{|S_{2}|-1,\,\text{if}\,S_{2}\neq \emptyset},\quad \underbrace{\frac{2|I_1|}{N}}_{1,\,\text{if}\,S_{1}\neq \emptyset\,\text{and}\,S_{2} \neq \emptyset}, \\
& \underbrace{\frac{2(-|S_3|+|S_4|)}{N}}_{|S_{3}|-1,\,\text{if}\,S_{3}\neq \emptyset},\quad \underbrace{\frac{2(|S_3|-|S_4|)}{N}}_{|S_{4}|-1,\,\text{if}\,S_{4}\neq \emptyset},\quad \underbrace{\frac{2|I_2|}{N}}_{1,\,\text{if}\,S_{3}\neq \emptyset\,\text{and}\,S_{4} \neq \emptyset},\;\; 0,\;\; 2\varepsilon.
\end{split}
\end{align}

If $\widehat{\boldsymbol{\varphi}}$ is of index-$1$, then in \eqref{mid-eig4}, the number of positive eigenvalues is $1$, which is \(2\varepsilon\), and the number of negative eigenvalues is $N-2$. 
Assume $\min\{|S_1|,|S_2|\}\neq 0$, then the positive eigenvalue $\frac{2|I_1|}{N}$ in \eqref{mid-eig4} must exist, which is impossible. Similarly,  $\min\{|S_3|,|S_4|\}\neq 0$ implies that the eigenvalue $\frac{2|I_2|}{N}$ must exist, which is also impossible. Therefore, $\min\{|S_1|,|S_2|\}=\min\{|S_3|,|S_4|\}=0$. 

If $\min\{|S_1|,|S_2|\}=\min\{|S_3|,|S_4|\}=0$, then in \eqref{mid-eig4}, there are no positive eigenvalues other than the positive eigenvalue \(2\varepsilon\), and the number of negative eigenvalues in \eqref{mid-eig4} is \(N - 2\). Therefore, $\widehat{\boldsymbol{\varphi}}$ is of index-$1$.

$(2)$ Given $|I_1|\neq |I_2|$ and  $0<\varepsilon<\min\{|I_1|,|I_2|\}/\big||I_1|-|I_2|\big|$.  If $\widehat{\boldsymbol{\varphi}}$ is of index-$1$, then there are no positive eigenvalues in  \eqref{mid-eig1}-\eqref{mid-eig2}. So, $\min\{|S_1|,|S_2|\}=\min\{|S_3|,|S_4|\}=0$. 

If $\min\{|S_1|,|S_2|\}=\min\{|S_3|,|S_4|\}=0$,   for the case where $\min\{|S_1|,|S_2|\}=0$, without loss of generality, assume $|S_1|=0$ and $|S_2|=|I_1|\ge 1$. By \eqref{mid-eig1}, the eigenvalues $\frac{2(-|S_1|+|S_2|)}{N}+\frac{2\varepsilon(|I_2|-|I_1|)}{N}$ and $\frac{2(|S_1|-|S_2|)}{N}+\frac{2\varepsilon(|I_2|-|I_1|)}{N}$ do not exist. Given $|S_1|=0$ and $|S_2|=|I_1|\ge 1$, and $0<\varepsilon<\min\{|I_1|,|I_2|\}/\big||I_1|-|I_2|\big|$, the negative eigenvalue $\frac{2|I_1|}{N}+\frac{2\varepsilon(|I_2|-|I_1|)}{N}$ exists with multiplicity $|S_2|-1$ \ (i.e., $|I_1|-1$) when $|S_2|\ge 2$; $\frac{2|I_1|}{N}+\frac{2\varepsilon(|I_2|-|I_1|)}{N}$ does not exist when $|S_2|=1$. 
Similarly, for $\min\{|S_3|,|S_4|\}=0$, the number of negative eigenvalues in \eqref{mid-eig2} is $|I_2|-1$, and no positive eigenvalues exists. Therefore, $\widehat{\boldsymbol{\varphi}}$ is index-$1$ saddle point of system \eqref{mod}.
\end{proof}

\begin{remark}
The condition $\min\{|S_3|, |S_4|\}=0$ in \eqref{eq:min4S} is  about the set  $I_2$ for the two  memorized binary pattern $\boldsymbol{\xi}^1$ and $\boldsymbol{\xi}^2$.
It means that   $\xi_i^1 \equiv$ the constant $1$  for all $i\in I_2$, or $\xi_i^1 \equiv$ the constant $-1$  for all $i\in I_2$. 
For such  memorized binary patterns, 
the other  condition $ \min\{|S_1|,|S_2|\}=0$ in   \eqref{eq:min4S}     gives only two index-1 saddle points satisfying \eqref{mid1} among 
memory-compatible ternary-pattern points, corresponding to $S_1=I_1$ (i.e., $|S_2|=0$) or $S_2=I_1$ (i.e., $|S_1|=0$), respectively,  with   the expressions    given by, 
\begin{equation}\label{posxi2}
\widehat{\varphi}_i:=
\begin{cases} 
0, & \mbox{ if } \xi_i^1=\xi_i^2=1;\\
\pi, & \mbox{ if } \xi_i^1=\xi_i^2=-1;\\
\pi/2, & \mbox{ if } \xi_i^1\ne \xi_i^2 ~(\mbox{i.e.,}~ i\in I_2).
\end{cases}
\quad \mbox { or } \quad 
\widehat{\varphi}_i:=
\begin{cases} 
\pi, & \mbox{ if } \xi_i^1=\xi_i^2=1;\\
0, & \mbox{ if } \xi_i^1=\xi_i^2=-1;\\
\pi/2, & \mbox{ if } \xi_i^1\ne \xi_i^2 ~(\mbox{i.e.,}~ i\in I_2).
\end{cases}   
\end{equation}
In fact, they are the midpoint  \(\boldsymbol{\varphi}_{1/2}=(\boldsymbol{\varphi}^{*}(\boldsymbol{\xi}^1)+\boldsymbol{\varphi}^{*}(\boldsymbol{\xi}^2))/2\)
and its image $\pi-\boldsymbol{\varphi}_{1/2}$.

There are similar results about the case \eqref{mid2}
for the memory-compatible ternary-pattern points, in parallel to  Theorem \ref{Morthm1} associated with the case \eqref{mid1}.  If   two memories $\boldsymbol{\xi}^1$ and $\boldsymbol{\xi}^2$ satisfies that $\xi_i^1 \equiv$   $1$  for all $i\in I_1$, or $\xi_i^1 \equiv$  $-1$ for all $i\in I_1$,
then  
there are only two index-1 saddle points satisfying \eqref{mid2}: 
\begin{equation}\label{negxi2}
\widehat{\varphi}_i:=
\begin{cases} 
0, & \mbox{ if }~ \xi_i^1=-\xi_i^2=1;\\
\pi, & \mbox{ if }~ \xi_i^1=-\xi_i^2=-1;\\
\pi/2, & \mbox{ if }~ \xi_i^1= \xi_i^2 ~(\mbox{i.e.,}~ i\in I_1) .
\end{cases}
\quad \mbox { or } \quad 
\widehat{\varphi}_i:=
\begin{cases} 
\pi, & \mbox{ if }~ \xi_i^1=-\xi_i^2=1;\\
0, & \mbox{ if }~ \xi_i^1=-\xi_i^2=-1;\\
\pi/2, & \mbox{ if } ~\xi_i^1=  \xi_i^2 ~(\mbox{i.e.,}~ i\in I_1).
\end{cases}  
\end{equation}
which are the midpoint  \(\boldsymbol{\varphi}_{1/2}=(\boldsymbol{\varphi}^{*}(\boldsymbol{\xi}^1)+\boldsymbol{\varphi}^{*}(\boldsymbol{-\xi}^2))/2\)
and its image $\pi-\boldsymbol{\varphi}_{1/2}$, where 
$\boldsymbol{\varphi}^*=\arccos$.

In summary,  there are at most  four index-1 saddle points among memory-compatible ternary-pattern points, corresponding to 
the midpoints of two memories; in addition, 
Since \eqref{eq:min4S} is sufficient and necessary, 
the existence of index-1 memory-compatible ternary critical points
also requires the condition for two memories $\boldsymbol{\xi}^1$,
$\boldsymbol{\xi}^2$.

\end{remark}

\section{Example}\label{sec5}
We show an example with $N=6$ for the system \eqref{mod} to 
illustrate the theoretic results we have obtained above.
The  two memorized binary patterns  are 
\[\boldsymbol{\xi}^1=( 1\ 1 \ 1 \ -1 \ -1 \ -1)^{\top}\quad \text{and} \quad \boldsymbol{\xi}^2=(1\ 1 \ 1 \ 1 \ 1 \ -1)^{\top}, \]
which are not  orthogonal.
$I_1=\{1,2,3,6\}$ and $I_2=\{4,5\}$  (see \eqref{index}). 
$S_1=I_1$ $S_2=\emptyset$, 
and  $S_4=I_2$, $S_3=\emptyset$ by \eqref{S1234}, then
\eqref{eq:min4S} in Theorem \ref{Morthm1}
holds.
$\varepsilon^*$ as defined in \eqref{epstar} equals to $1/3$ and  the upper bound $\widehat{\varepsilon}$ in Theorem \ref{Morthm1} is $1$.
We choose $\varepsilon=0.3 < \min\{\varepsilon^*,\widehat{\varepsilon}\}$.
Note that the Morse indices 
do not depend on the value of $\varepsilon$ as long as it is between $0$ and $1/3$.

In this example, we   focus on $3^6=729$ ternary-pattern points of system \eqref{mod}, which include the symmetric pair of $\boldsymbol{\varphi}$ and $\pi-\boldsymbol{\varphi}$. Firstly,
through numerical calculation of the right hand side of \eqref{mod},
the total ternary-pattern points consist of $614$ non-critical points and $115$ critical points.
These critical points are consistent with the prediction of 
Proposition \ref{prop}.
Among these $115$ ternary-pattern critical points, 
we summarize their Morse indices and  compute the energy levels below in details.
\begin{enumerate}
	\item[$(1)$] Binary-pattern critical points ($2^N=64$ points):  
	\begin{itemize}
		\item Index-$0$
		(stable): {\it Only the memorized binary-pattern critical points   are stable due to  Theorem \ref{lem2}. They are $4$ points: \(\pm \boldsymbol{\xi}^1\) and \(\pm\boldsymbol{\xi}^2\),  with the same  energy of $-3.7833$.
			They are shown (with their sign-inverted images) at the lowest level of Figure \ref{fig:example}.}

		\item Index-$1$: {\it There are $32$ points of index-$1$, of which $12$ points have  an energy  of $-1.1167$, $16$ points have an energy of $-1.7833$ and $4$ points have an energy of $-3.1167$. They can be found by checking the conditions in  Remark \ref{rem2}. These 16 points 
			(by neglecting  their sign-inverted images) are highlighted by the three dashed   boxes in Figure \ref{fig:example}.}
		
		\item Index-$2$: { There are  $28$ points of index-$2$, including $12$ points with an energy of $-0.4500$ and $16$ points with an energy of $-1.1167$. This is from Theorem \ref{thm1}.}
	\end{itemize}
	
	\item[$(2)$]  Memory-compatible ternary-pattern critical points ($2^{|I_1|}+2^{|I_2|}=20$ points):
	
	\begin {enumerate}
	\item[$(2.1)$]  satisfying  \eqref{mid1}  ($2^{|I_1|}=16$ points):
	\begin{itemize}
		\item Index-$1$: 
		{\it There are $2$ points of index-$1$, the midpoint $\boldsymbol{\varphi}_{1/2}=(0,0,0,\pi/2, \pi/2,\pi)$ of \(\varphi^*(\boldsymbol{\xi}^1)\) and \(\varphi^*(\boldsymbol{\xi}^2)\) and its image $\pi-\boldsymbol{\varphi}_{1/2}=(\pi,\pi,\pi,\pi/2, \pi/2,0)$, with an energy of $-3.3833$ for each.

			This result is from Theorem \ref{Morthm1}.
			This index-1 saddle point and its image  have the lowest energy among all saddle points and highlighted in the middle   box of   Figure \ref{fig:example}. This saddle point connects to the two memories by following its one dimensional unstable manifold. So, we can regard this index-1 saddle point as the {transition state}.}
		\item Index-$2$: There are $14$ points, including $8$ points with an energy of $-1.3833$ and $6$ points with an energy of $-0.7167$. This is from Theorem \ref{thm-mid}. 
	\end{itemize}
\end{enumerate}

\begin {enumerate}
\item[$(2.2)$]  satisfying  \eqref{mid2}  ($2^{|I_2|}=4$ points):
\begin{itemize}
	\item Index-$1$: non-existence,       because of  $\min\{|S_3|,|S_4|\}>0$ in \eqref{eq:min4S} and Theorem \ref{Morthm1}  after changing 
	$\boldsymbol{\xi}^2\to -\boldsymbol{\xi}^2$.
	\item Index-$2$:  There are $2$ points with an energy of $-1.3833$ for each.
	\item Index-$3$: There are $2$ points with an energy of $-0.7167$ for each.
\end{itemize}
\end{enumerate}

\item[$(3)$]  Other ternary-pattern critical points ($31$ points):   not covered by our theoretic results and we need to  numerically check their eigenvalues.
\begin{itemize}
\item Index-$1$: 
{\it There is only one index-1 point $\frac{\pi}{2}\mathbbm{1}$ with the same energy of $-1.7833$ as some index-1 binary-pattern points.  
	But this should  be regarded  equivalent to the binary-pattern index-1 saddle point $\pi\mathbbm{1}$ (the first one shown  in the right box of Figure \ref{fig:example}) by global phase shift of $\pi/2$.}
\item Index-$3$: There are $18$ points with an energy of $-0.7167$ for each.
\item Index-$4$: There are $12$ points with an energy of $-0.0500$ for each.
\end{itemize}

\end{enumerate}

In Figure \ref{fig:example}, we depict all  index-$0$ and index-$1$ saddle points, where the visual height is related to the energy level. 
The paths from saddle points to local minimum points are marked too.
The `Index' and `Energy' values at the top of each pattern represent the Morse index and the potential energy of the pattern, respectively. 
We adopt  the energy landscape viewpoint \cite{FW2012,Energylanscapes2024ReviewJCP} to interpret this figure. 
This figure demonstrates the transition dynamics between two stable memory patterns 
$\boldsymbol{\xi}^1$ and $\boldsymbol{\xi}^2$ (including their symmetric images $-\boldsymbol{\xi}^1$ and $-\boldsymbol{\xi}^2$). 
These index-1 saddle points lie at four distinct levels of potential energy, and the lowest energy $V=-3.3833$ corresponds to the midpoints of the two memorized binary patterns ($V=-3.7833$).   This special index-1 saddle point with the lowest energy connects the two memories via its unstable manifold, so 
it is surely the transition state between the stable memories, with the activation energy (the energy difference) $\Delta V=0.4$.  
The six index-1 saddle points with the highest energy at the top of Figure \ref{fig:example} also connect two memories but with a much higher energy barrier $\Delta V=2.6666$, suggesting a lower chance of escaping one of the two stable memories from its stability region via passing through these saddle points. The remaining index-1 saddle points only lie on the separatrix of one memory and its symmetric image, and are irrelevant to the transition between two memories.  
This  exploration of these index-1 saddle points, their energy level, and their unstable manifolds enhances our understanding of the phase space and dynamics of the Kuramoto model \eqref{mod}. The concept of transition states,  the most relevant critical point to transition between bi-stable states in the double-well potential, might have implications in biological meaning, when the brain is subject to random perturbation.

 \begin{figure}[htbp]
	\centering
	\includegraphics[width=1\textwidth]{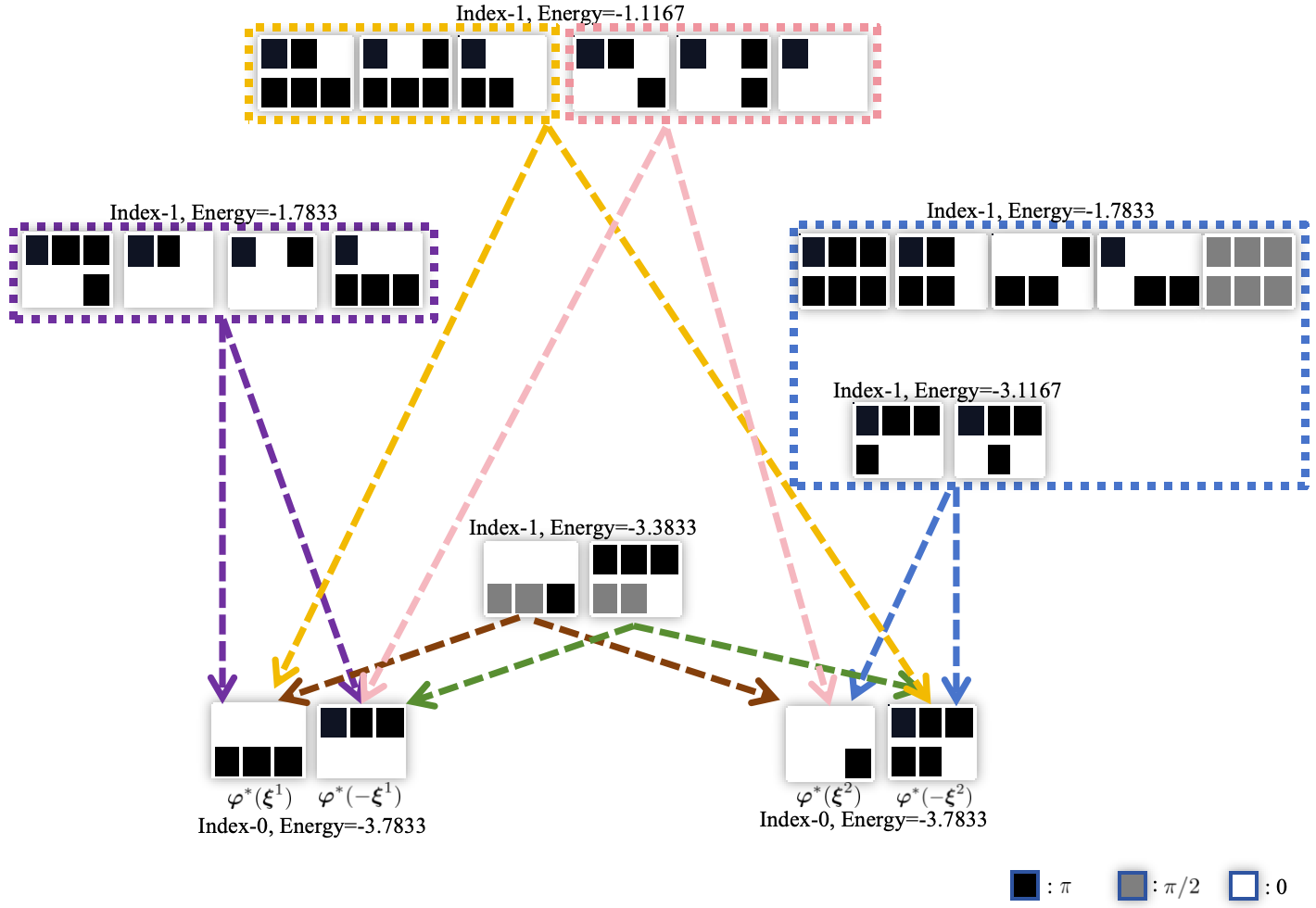} 
	\caption{The image depicts index-$0$ and index-$1$ critical points of ternary patterns. The `Index' and `Energy' values at the top represent the Morse index and potential energy, respectively. The arrows represent the unstable manifolds  originating from the index-1 saddle points, with the system's solution eventually converging to one of  local stable patterns.  Except the two stable points and the index-1 saddle point with lowest energy (energy=$-3.3833$),  the symmetric  images ($\boldsymbol{\varphi}\to \pi-\boldsymbol{\varphi}$) of other critical points are not shown.   }
	\label{fig:example} 
\end{figure}

\section{Conclusion}\label{secon}
In conclusion, this study comprehensively explores the Kuramoto model of two oscillators with  the second-order Fourier coupling and  the Hebbian learning rule. By detailed analysis of the critical points of binary patterns and the ternary patterns, we have obtained a relatively complete characterization    regarding their Morse index and particularly the index-1 saddle points. The determination of the spectra and Morse indices of these points offers a more comprehensive understanding of the system's dynamics, beyond the locally stable patterns.   Future research can build on these findings to further explore the behavior of the system under more complex conditions. For example, we can consider more complex coupling mechanisms or a larger number of oscillators. 

It is known that many dynamical systems based on  Hebbian rule can exhibit multiple meta-stable states, as shown in the  Kuramoto model with second-order Fourier coupling studied here. The associative memory relies on the partition of the phase space of the input signals into different basins of attraction of each memory. By incorporating the existing concepts of energy landscape and transitions prevailing in statistical physics\cite{FW2012}, theoretical chemistry\cite{Energylanscapes}, and biology\cite{WANGPNAS2011}, one may gain a deeper understanding of the underlying dynamical models beyond the linear stability regime and pursue  further investigation of non-equilibrium properties related to the storage and retrieval of the memory in brain science.
The example here shows how to describe the index-1 saddle points based on theoretical analysis for a particular dynamical system.  For more general and complex bistable systems, the numerical methods for index-1 saddle point search  \cite{GAD2011,String2002,SimGAD2018,Dimer1999,yin2019high} can play a major role in practice.


\vspace{0.2cm}


 \textbf{Data Accessibility.} This article has no additional data.

\textbf{Declaration of AI use.} {We have not used AI-assisted technologies in creating this article.}

\textbf{Authors' Contributions.} {X. Zhao: conceptualization, formal analysis, funding acquisition, investigation, methodology, writing-original draft; X. Zhou: conceptualization, methodology, supervision, validation, visualization, writing-review\&editing, funding acquisition.
	All authors gave their final approval for publication and agreed responsible for the work performed therein.}

\textbf{Competing Interests.} {We declare we have no competing interests.}

\textbf{Funding.} {This work is partially supported by Hong Kong RGC GRF grants  11318522 and 11308323. Zhao acknowledges the National Natural Science Foundation of China (Grant No. 12201156), the support provided by the Hong Kong Scholars Scheme (Grant No. XJ2023001) and the Fundamental Research Funds for the Central Universities.}


\bibliographystyle{plain}

\bibliography{ref}

\begin{thebibliography}{10}

\bibitem{AonishiPRE1998}
Toru Aonishi.
\newblock Phase transitions of an oscillator neural network with a standard
  hebb learning rule.
\newblock {\em Phys. Rev. E}, 58:4865--4871, Oct 1998.

\bibitem{choi2021asymptotic}
Y.-P. Choi, S.-Y. Ha, Q.~Xiao, and Y.~Zhang.
\newblock Asymptotic stability of the phase-homogeneous solution to the
  {Kuramoto-Sakaguchi} equation with inertia.
\newblock {\em SIAM J. Math. Anal.}, 53(3):3188--3235, 2021.

\bibitem{dong2013synchronization}
J.-G. Dong and X.~Xue.
\newblock Synchronization analysis of {Kuramoto} oscillators.
\newblock {\em Commun. Math. Sci.}, 11(2):465--480, 2013.

\bibitem{dorfler2011critical}
F.~D{\"o}rfler and F.~Bullo.
\newblock On the critical coupling for {Kuramoto} oscillators.
\newblock {\em SIAM J. Appl. Dyn. Syst.}, 10(3):1070--1099, 2011.

\bibitem{Duannurev2025}
Matthew Du, Agnish~Kumar Behera, and Suriyanarayanan Vaikuntanathan.
\newblock Physical considerations in memory and information storage.
\newblock {\em Annual Review of Physical Chemistry}, 76(Volume 76,
  2025):471--495, 2025.

\bibitem{String2002}
W.~E, W.~Ren, and E.~Vanden-Eijnden.
\newblock String method for the study of rare events.
\newblock {\em Phys. Rev. B}, 66:052301, 2002.

\bibitem{TPT_review}
Weinan E and Eric Vanden-Eijnden.
\newblock Transition path theory and path-finding algorithms for the study of
  rare events.
\newblock {\em Annu. Rev. Phys. Chem.}, 61:391--420, 2010.

\bibitem{GAD2011}
Weinan E and Xiang Zhou.
\newblock The gentlest ascent dynamics.
\newblock {\em Nonlinearity}, 24(6):1831, 2011.

\bibitem{Erying}
H.~Eyring.
\newblock The activated complex and the absolute rate of chemical reactions.
\newblock {\em Chem. Rev.}, 17:65--77, 1935.

\bibitem{6879338}
Rosangela Follmann, Elbert E.~N. Macau, Epaminondas Rosa, and José R.~C.
  Piqueira.
\newblock Phase oscillatory network and visual pattern recognition.
\newblock {\em IEEE Transactions on Neural Networks and Learning Systems},
  26(7):1539--1544, 2015.

\bibitem{FW2012}
M.~I. Freidlin and A.~D. Wentzell.
\newblock {\em Random Perturbations of Dynamical Systems}.
\newblock Grundlehren der mathematischen Wissenschaften. Springer-Verlag, New
  York, 3 edition, 2012.

\bibitem{SimGAD2018}
Shuting Gu and Xiang Zhou.
\newblock Simplified gentlest ascent dynamics for saddle points in non-gradient
  systems.
\newblock {\em Chaos: An Interdisciplinary Journal of Nonlinear Science},
  28(12):123106, 2018.

\bibitem{ha2013formation}
S.-Y. Ha, Z.~Li, and X.~Xue.
\newblock Formation of phase-locked states in a population of locally
  interacting {Kuramoto} oscillators.
\newblock {\em J. Differ. Equ.}, 255(10):3053--3070, 2013.

\bibitem{hebb1949the}
D.~O. Hebb.
\newblock {\em The Organization of Behavior}.
\newblock Wiley, New York, 1949.

\bibitem{Dimer1999}
G.~Henkelman and H.~J\'{o}nsson.
\newblock A dimer method for finding saddle points on high dimensional
  potential surfaces using only first derivatives.
\newblock {\em J. Chem. Phys.}, 111(15):7010--7022, 1999.

\bibitem{hölzel2015stability}
R.~W. H{\"o}lzel and K.~Krischer.
\newblock Stability and long term behavior of a {Hebbian} network of {Kuramoto}
  oscillators.
\newblock {\em SIAM J. Appl. Dyn. Syst.}, 14(1):188--201, 2015.

\bibitem{Hopfield1982neural}
J.~J. Hopfield.
\newblock Neural networks and physical systems with emergent collective
  computational abilities.
\newblock {\em Proc. Natl. Acad. Sci. U.S.A.}, 79(8):2554--2558, 1982.

\bibitem{Hopfield1984neurons}
J.~J. Hopfield.
\newblock Neurons with graded response have collective computational properties
  like those of two-state neurons.
\newblock {\em Proc. Natl. Acad. Sci. U.S.A.}, 81(10):3088--3092, 1984.

\bibitem{SAOPT}
S.~Kirkpatrick, C.~D. Gelatt, and M.~P. Vecchi.
\newblock Optimization by simulated annealing.
\newblock {\em Science}, 220(4598):671--680, 1983.

\bibitem{kuramoto1984chemical}
Y.~Kuramoto.
\newblock {\em Chemical Oscillations, Waves and Turbulence}.
\newblock Springer, 1984.

\bibitem{TUPNAS2024Review}
Herbert Levine and Yuhai Tu.
\newblock Machine learning meets physics: A two-way street.
\newblock {\em Proceedings of the National Academy of Sciences},
  121(27):e2403580121, 2024.

\bibitem{li2022hebbian}
Z.~Li, X.~Zhao, and X.~Xue.
\newblock Hebbian network of {Kuramoto} oscillators with second-order couplings
  for binary pattern retrieve: {II}. nonorthogonal standard patterns and
  structural stability.
\newblock {\em SIAM J. Appl. Dyn. Syst.}, 21(1):102--136, 2022.

\bibitem{L-Z-Z}
Z.~Li, X.~Zhao, and X.~Zhou.
\newblock Reduce the multi-stability in associative-memory network using two
  memories and an extended approach for error-free retrieval.
\newblock 2025.

\bibitem{monzon2005global}
P.~Monzon and F.~Paganini.
\newblock Global considerations on the {Kuramoto} model of sinusoidally coupled
  oscillators.
\newblock In {\em Proceedings of the 44th IEEE Conference on Decision and
  Control}, pages 3923--3928. IEEE, 2005.

\bibitem{nishikawa2004oscillatory}
T.~Nishikawa, F.~C. Hoppensteadt, and Y.-C. Lai.
\newblock Oscillatory associative memory network with perfect retrieval.
\newblock {\em Phys. D}, 197(1-2):134--148, 2004.

\bibitem{nishikawa2004capacity}
T.~Nishikawa, Y.-C. Lai, and F.~C. Hoppensteadt.
\newblock Capacity of oscillatory associative-memory networks with error-free
  retrieval.
\newblock {\em Phys. Rev. Lett.}, 92(10):108101, 2004.

\bibitem{Pollak2005Rev}
Eli Pollak and Peter Talkner.
\newblock Reaction rate theory: what it was, where is it today, and where is it
  going?
\newblock {\em Chaos: An Interdisciplinary Journal of Nonlinear Science},
  15(2):026116, 2005.

\bibitem{Energylanscapes2024ReviewJCP}
J~C Sch{\"o}n.
\newblock {Energy landscapes{\textemdash}Past, present, and future: A
  perspective}.
\newblock {\em The Journal of Chemical Physics}, 161(5):050901, 2024.

\bibitem{Energylanscapes}
D.~J. Wales.
\newblock {\em Energy Landscapes with Application to Clusters, Biomolecules and
  Glasses}.
\newblock Cambridge University Press, 2003.

\bibitem{WANGPNAS2011}
Jin Wang, Kun Zhang, Li~Xu, and Erkang Wang.
\newblock Quantifying the waddington landscape and biological paths for
  development and differentiation.
\newblock {\em Proceedings of the National Academy of Sciences},
  108(20):8257--8262, 2011.

\bibitem{Wigner}
E.~Wigner.
\newblock The transition state method.
\newblock {\em Trans. Faraday Soc.}, 34:29--41, 1938.

\bibitem{LossPDE2023}
Keke Wu, Xiangru Jian, Rui Du, Jingrun Chen, and Xiang Zhou.
\newblock Roughness index for loss landscapes of neural network models of
  partial differential equations*.
\newblock In {\em 2023 IEEE International Conference on Big Data (BigData)},
  pages 966--975, 2023.

\bibitem{yin2019high}
Jianyuan Yin, Lei Zhang, and Pingwen Zhang.
\newblock High-index optimization-based shrinking dimer method for finding
  high-index saddle points.
\newblock {\em SIAM Journal on Scientific Computing}, 41(6):A3576--A3595, 2019.

\bibitem{npjLeiZhang2016}
Lei Zhang, Weiqing Ren, Amit Samanta, and Qiang Du.
\newblock Recent developments in computational modelling of nucleation in phase
  transformations.
\newblock {\em npj Computational Materials}, 2(1):16003, 2016.

\bibitem{zhao2025binary}
X.~Zhao and Z.~Li.
\newblock Binary pattern retrieval with {Kuramoto-type} oscillators via a least
  orthogonal lift of three patterns.
\newblock {\em Eur. J. Appl. Math.}, 36(2):448--463, 2025.

\bibitem{zhao2020stability}
X.~Zhao, Z.~Li, and X.~Xue.
\newblock Stability in a {Hebbian} network of {Kuramoto} oscillators with
  second-order couplings for binary pattern retrieve.
\newblock {\em SIAM J. Appl. Dyn. Syst.}, 19(2):1124--1159, 2020.

\bibitem{zhao2023unified}
X.~Zhao, Z.~Li, and X.~Xue.
\newblock Unified approach for applications of oscillatory associative-memory
  networks with error-free retrieval.
\newblock {\em Phys. Rev. E}, 108(1):014305, 2023.

\bibitem{zhao2025first}
X.~Zhao, Z.~Li, X.~Xue, and Y.~Zhao.
\newblock First-order euler method is effective for computation of
  associative-memory network of {Kuramoto} oscillators.
\newblock {\em Commun. Nonlinear Sci. Numer. Simul.}, page 108650, 2025.

\end{thebibliography}

\end{document}